\documentclass[12pt]{amsart}
\usepackage{amsmath, amssymb, amsthm, verbatim,enumerate,bbm}
\usepackage{indentfirst}
\usepackage{listings}
\usepackage{color}
\usepackage{graphicx}
\usepackage[eulergreek]{sansmath}
\usepackage{pgfplots, pgfplotstable}
\newcommand{\PRLsep}{\noindent\makebox[\linewidth]{\resizebox{1.2\linewidth}{1pt}{$\bullet$}}\bigskip}
\usepackage{mathrsfs}
\usetikzlibrary{arrows}

\usepackage{tikz} 
 \usetikzlibrary{calc} 
\DeclareMathOperator{\Tr}{Tr}
\date{\today}
\allowdisplaybreaks

\theoremstyle{plain}
\newtheorem{theorem}{Theorem}[section]
\newtheorem{lemma}[theorem]{Lemma}
\newtheorem{example}[]{Example}

\newtheorem{coro}[theorem]{Corollary}

\newtheorem{remark}[theorem]{Remark}
\newtheorem{definition}[theorem]{Definition}

\newcommand{\ve}{\varepsilon}

\newcommand{\BP}{\mathbf P} 
\newcommand{\BE}{\mathbf E}

\newcommand{\cK}{\mathcal K}
\newcommand{\dist}{dist}
\newcommand{\cM}{\mathcal M}
\newcommand{\cW}{\mathcal W}
\tikzstyle{vertex}=[circle,fill=black!25,minimum size=5pt,inner sep=0pt]
\tikzstyle{every node}=[circle, draw, fill=black!50,
                        inner sep=0pt, minimum width=4pt]

\begin{document}

\title{Spectrum of complex networks}

\author[D. Montealegre]{Daniel Montealegre}
\address{Department of Mathematics, Yale University }
\email{daniel.montealegre@yale.edu}
\author[V. Vu]{Van Vu}
\address{Department of Mathematics, Yale University}
\email{van.vu@yale.edu}
\begin{abstract}
	The study of complex networks has been one of the most active fields in science in recent decades. Spectral properties of networks (or graphs that represent them) are of fundamental  importance. 
	Researchers have been investigating these properties for many years, and, based on numerical data, 
	have raised a number of questions about the distribution of the eigenvalues and eigenvectors. 
	
	In this paper, we give the solution to some of these questions. In particular, we 
	determine the limiting distribution of (the bulk of) the spectrum as the size of the network grows to infinity and  show that the leading eigenvectors are strongly localized.  
	
	We focus  on the preferential attachment graph, which is the most popular
	mathematical  model for growing complex networks. Our analysis is, on the other hand,  general and can be applied to other models. 
	
\end{abstract}

\maketitle 
\section{Introduction}
\noindent The study of complex networks has been an extremely active field in recent years (e.g., \cite{AJB}, \cite{A}, \cite{BAJ}, \cite{HH}, \cite{PPV}, \cite{PWKO}, \cite{XLC}). A   fundamental goal of this study is to build  and study  mathematical models of networks that occur in nature. 
\\
A popular model in the field is  the (random) preferential attachment model. In this model, the  network grows continuously. At time $t=1,2,3, \dots$, a  new vertex is born and attaches itself randomly to existing vertices, with respect to a distribution which favors vertices with higher degrees. This reflects the famous ``rich get richer" phenomenon observed in real-life situations. As pointed out by many researchers, this model of random graphs is significantly different from the standard Erd\H{o}s-R\'{e}nyi model where all vertices play the same role. There is vast literature on this model and  we refer the reader to the survey \cite{Bollobas} as a starting point; see also \cite{BR, AB, Tusnady, PRR} and the references therein. 
\\

For any network (in fact for any data set which can be represented in matrix form), its spectral information is of fundamental interest. Based on numerical experiments, researchers have observed the following about the eigenvalues and eigenvectors of  PA graphs:

{\bf Observation 1.} The bulk of the spectrum has triangular shape. This is very different from the classical semi-circle law by Wigner which holds for Erd\H{o}s-R\'{e}nyi graph (e.g., \cite{AB}, \cite{EG},  \cite{FDBV}, \cite{PJ}). The edge eigenvalues seem to follow a power law. See Figure \ref{hist}. 
\begin{figure}
	\begin{center}
		\includegraphics[scale = 0.25]{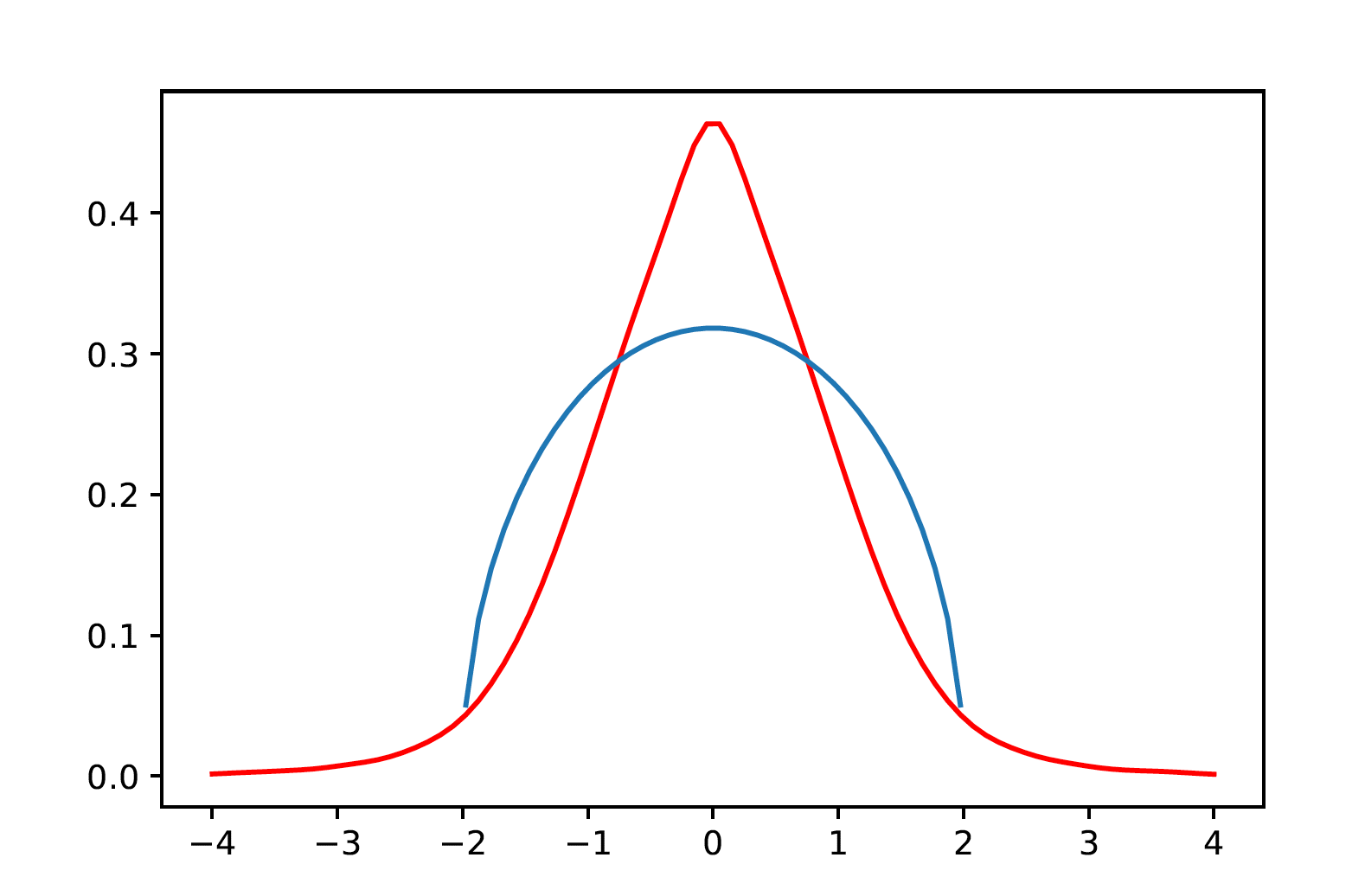}
	\end{center}
	\caption{In red: The spectral measure of a random PA graph with $n = 6000$, $m = 5$ with a normalization factor of $1/ \sqrt{n p(1-p)}$ where $p = 2m/n$. In blue: the semicircle distribution $(1/2\pi) \sqrt{4 - x^2}$ for $x\in [-2,2]$.}\label{hist}
\end{figure}

{\bf Observation 2.} The leading eigenvectors are localized, namely, there are a few  coordinates which capture a large amount of mass (e.g.,\cite{AB},  \cite{MZN}, \cite{FDBV}).

{\bf Heuristics 3.} The following heuristic has  been mentioned  in several papers \cite{AB, FDBV}: the power law of the edge eigenvalues is due to the localization of the corresponding eigenvectors.
\\

\vskip2mm
 These observations became well-known about twenty years ago. Among others, they have been discussed in 
Albert and Barabasi's widely circulated survey, which has been cited close to twenty thousand times; 
see Section 7 of \cite{AB}. However,  there have been  few rigorous  results.  The most relevant paper is  \cite{FFF}, which shows that for any fixed $k$, the first $k$ eigenvalues of 
the graph satisfy a power law.   More is known about 
static models where the number of vertices is fixed and the edges are drawn independently (but with different probablities). However, from the spectral point of view, these models seem to behave differently from the 
PA model; see \cite{CLV, CLV1} and the references there in. 

In \cite{Chayes}, it has been showed that the PA graph has a weak limit, which is a random tree. However, while the definition of the random tree is explicit, 
it seems difficult to compute the limiting spectral distribution of this tree or even its moments. In particular, it is 
already a highly non-trivial task  to prove that this limiting  distribution has a finite number of atoms; see \cite{Bon1, Sal} for more discussion. On the other hand, based
 on the numerical experiment, it is reasonable to conjecture that  there is no atom, but this is still open. \footnote {We  would like to thank J. Salez for pointing out these references.} 

 In this paper, we report  recent theoretical progress in the understanding of the bulk of the spectrum and the localization of leading eigenvectors. In particular, we
determine the moments of the limiting measure up  to any precision. Next, we prove 
the localization phenomenon in a precise form, determining the exact mass of the localized coordinate.
Our proofs also provide a satisfying explanation for {\bf Heuristics 3}.

\section{New results}

\subsection {The model} \label{model}

Let us first give  the precise description of a PA graph, following \cite{BR, Bollobas}. The key parameter here is a constant $m$, which is the number of links (edges) from a new vertex to the existing ones. Intuitively, one wants to link  the new vertex  to any  existing vertex $u$ (including itself) with probability proportional to the degree of $u$. The precise definition requires some care, since the degrees 
change after each link has been made.
\\
\\
\noindent First start with  $m=1$. The graph $G_{1,1}$  consists of $1$ loop around vertex $1$. Recursively define $G_{1,t}$ from $G_{1,t-1}$ as follows. Add a new  vertex $t$
(so our vertex set at time $t$ will be $[t]:=\{1,2, \dots, t\}$, and connect it to a random point $X_t \in [t]$ chosen with the following distribution
$$ 
\BP[X_t = i] = 
\begin{cases}
\frac{d(i, G_{1,t-1})}{2t - 1}  &\text{ if $1\leq i< t$}\\
\frac{1}{2t-1}  &\text{ if $i = t$}
\end{cases}
$$where $d(u, G)$ is the degree of $u$. Note that this definition allows for loops. 
\\
\\
\noindent Now we consider arbitrary $m$.  We construct $G_{m,t}$ by partitioning the vertices 
$G_{1,mt}$ into $t$ consecutive groups of size $m$ and viewing each group as a vertex of $G_{m,t} $.
Technically speaking, the  vertices $\{(a-1)m+1, (a-1)m+2,\ldots, am\}$ of $G_{1, mt}$ get collapsed into  vertex  $a$ of $G_{m,t}$; for $a=1,\dots, t$. This is equivalent to adding $m$ edges at each time step where those edges are added one by one, counting  the previous edges as well as the ``half" contribution to the degrees.  We denote by  $G_{m,t}$ the graph at time $t$. For more details about this model, see \cite{Bollobas}. The asymptotic notation such as $O,\Theta, o$ are used under the assumption that the size of the graph tends to infinity.

\subsection{The moments of the adjacency matrix}

As the size of the graph tends to infinity, the most natural question is to find the limit of its spectrum.
To be specific, at time $n$, the adjacency matrix of the graph has $n$ eigenvalues $\lambda_1(n), \dots, \lambda_n(n) $. We generate the spectral measure  $\mu_n$ on the real line by defining $\mu_n (I)= \frac{1}{n} |\{ i, \lambda_i (n) \in I \} |$ for any interval $I$. The question is to find the limit of $\mu_n$ (after a 
possible rescaling if necessary) as $n$ tends to infinity. This limit, if it exists, is refereed to as  the limiting spectral measure.
\\
\\
\noindent  A natural (and popular)
method to determine the spectral measure is to compute its moments. Wigner famously used 
this method to prove his classical semi-circle law \cite{Wigner}.   Our first result provides the asymptotics of the moments in question.

\begin{theorem}\label{moments thm intro}Let $G_{m,n}$ be a PA random graph, and  $\mu_{m,n}$ its spectral measure. Let $C_k$ denote the $k$-th spectral moment of $\mu_{m,n}$. Then
	\begin{align*} 
	C_0 &= 1 \\
	C_1 &= \Theta\left(\frac{\log n}{n}\right) \\ 
	C_2 &= (1+ o(1)) 2m \\
	C_3 &= \Theta\left(1/\sqrt{n}\right) \\ 
	C_4 &= (1+ o(1)) 2m(m+1) \log n.
	\end{align*} Furthermore, for  $k \geq 3$ 
	\begin{align}
	C_{2k}&= \Theta(n^{k/2 - 1})\\
	C_{2k-1} &= \Theta(n^{k/2 - 3/2}).
	\end{align}
\end{theorem}

\noindent Despite this asymptotic result, one cannot use the moment method as in Wigner's case. The sequence $\{C_k\}$ 
simply {\it does not}  determine a distribution.  Since $C_2$ is  a constant (as $m$ is a constant), one does not need to do any rescaling. But without a rescaling, the higher 
moments tend to infinity.
\\
\\
\noindent The reason behind this surprising  fact is that the spectrum of $G_{m,n}$ consists of  two separate parts. The first part is formed by the edge eigenvalues and tends to follow a power law. The bulk of the spectrum (the triangle shape part) follows a different  law. The moments are dominated 
by the edge eigenvalues (which are significantly larger than the rest)  and thus reveal no information about the distribution of the eigenvalues in the bulk.  

\subsection{Approximate measures} 
Clearly, one needs a new idea to carry out further studies (to get information about the bulk of the spectrum, in particular).  The idea we propose here is to study approximations of the spectral measure, rather than the 
measure itself. 
\\
\\
\noindent To this end, we define the distance between two probability measure $\mu$ and $\eta$ on the real line as 
$$\dist (\mu, \eta) = \sup _I | \mu (I) -\eta (I) | , $$ 
where $I$ runs over the set of all  intervals. 
\\
\\
\noindent For any given precision $\ve >0$, we can define a sequence of (random)
measures $\mu_{\ve, m,n}  $ and a deterministic measure $\mu_{\ve, m,\infty} $ such that with probability one, $\dist (\mu_{m,n}, \mu_{\ve,m,n}  ) \le \ve $ and $ \mu_{\ve, m,n}  \rightarrow \mu_{\ve, m,\infty}. $ 
We define the approximation $\mu_{\ve,m,n}$ as the  spectral measure of  $G_{\ve,m,n}$, where 
$G_{\ve,m,n}$  is  the graph obtained by discarding the first $\ve n $ vertices of $G_{m,n}$ (we will refer to $G_{\ve,m,n}$ as the truncated graph). We prove the following

\begin{theorem} \label{theorem:limit} With probability one, 
	
	\begin{equation} \label{distance1} \dist (\mu_{m,n},  \mu_{\ve,m,n}) \le \ve. \end{equation}
	
	\noindent There  is a deterministic measure  $\mu_{\ve,m,\infty}$ (uniquely defined by 
	$\ve$ and $m$) such that  $\{\mu_{\ve,m,n}\}$ converges weakly in probability to $\mu_{\ve,m,\infty}$. 
	The moments of this limit can be computed explicitly for any given $\ve$. 
\end{theorem}

\noindent We designed the approximation scheme  to 
avoid the troublesome large eigenvalues.
The large eigenvalues in  the spectrum are determined by 
the largest degrees in the graph, which, naturally, come from the vertices  born earlier in the process.  We eliminate exactly these vertices. 
\\
\\
\noindent The estimate \eqref{distance1} is an immediate consequence of the interlacing law. The heart of the theorem is the existence and uniqueness of the limiting measure and the computation of its moments, which we now turn to.

\subsection{ The approximate limit and its moments} \label{the approximate limit}

We can  write down the approximate limit   $\{\mu_{\ve,m,\infty}\}$ through its moments. This distribution is symmetric around $0$, so all odd moments are zero. For each even number $k$, the $k$th moment is a constant 
$C(k, \ve, m)$, which depends only on  $k, \ve$ and $m$. To give an explicit expression for this quantity, we need some preparation. 
\\
\\
\noindent For the next discussion,  it will be convenient to think of the edges in the graph as directed, with the direction going from the larger end to the smaller end. 
(e.g., if $\{1,3\} $ is an edge, then the direction goes from $3$ to $1$.)  Let $G$ be a graph on $V(G)$, we denote the in-degree of a vertex $u$  by $d_{in}(u, G)$, and the out degree by $d_{out}(u, G)$. The set of vertices $u\in V(G)$ such that $d_{in}(u, G) >0$ is denoted by $V^{-}(G)$ (the in-vertices of $G$). Similarly, we define $V^{+}(G)$ as the out-vertices of $G$. Henceforth, when we write an edge as $(i,j)$ we tacitly imply $i\leq j$.

\begin{definition}Denote by $\cK_{m,n}$ the multigraph on vertex set $[n]$ where between any $i\neq j$ there are $m^2$ parallel edges, and each vertex has  $m(m+1)/2$ loops. 
\end{definition}

\noindent An essential technicality with the PA model is the order in which the  vertices appear. 

\begin{definition}\label{ordered}Define an ordered graph on $t$ vertices, to be a graph $H$ on vertex set $\{v_1,\ldots, v_t\}$ where we impose the ordering $v_i < v_{i+1}$. Note that we do not make the vertex set the integers since we want to emphasize this graph is not any specific subgraph of $\cK_{m,n}$. A subgraph of $\cK_{m,n}$ is said to be isomorphic to $H$ if it is isomorphic in the classical sense and  the mapping preserves the ordering. 
\end{definition}

\begin{definition} \label{walk1} Let $G$ be a multigraph. A walk of length $k$ is a sequence: $(v_1, e_1, v_2, \ldots, e_k, v_{k+1})$ where $v_i$ are vertices of $G$, $e_i$ are edges of $G$, and $e_i$ is incident to $v_i$ and $v_{i+1}$. If $v_{k+1} = v_1$, we call it a closed walk of length $k$. Let $\mathcal W_k(G)$ denote the set of all closed walks in $G$ of length $k$. 
\end{definition}
\noindent For a connected subgraph $H$ of $\mathcal K$, let $\cM_k(H)$ denote the number of elements of $\cW_k(\mathcal K)$ which yield $H$ as a subgraph.

\begin{definition}Let $\mathcal T_k$ denote the set of ordered trees with at most $k$ edges. 
\end{definition}

\begin{definition}Let $D =(d_1,\ldots d_t)$ be a sequence of positive integers with $\sum_i d_i = 2(t-1)$. Let $\ve > 0$ be a fixed constant between zero and one. Define
	\begin{equation}\label{psi}
	\psi(D, \ve) =  \frac{1}{(2m)^{t-1}} \int_{\ve}^1 \int_{\ve}^{y_t} \ldots \int_{\ve}^{y_2} \prod_{i=1}^t \frac{1}{y_i^{d_i/2}} dy_1\ldots dy_t. \end{equation}

\end{definition}

\noindent For any given sequence $D$, it is simple to compute  $\psi(D, \ve)$ using standard calculus.
Finally, for any graph $H$, set 

\begin{equation} \label{varphi} 
\varphi(H, m) = \left( \prod_{v\in V^{-}(H)} [m]^{d_{in}(v,H)}\right) \left(\prod_{v\in V^{+}(H)} [m]_{d_{out}(v, H)}\right), \end{equation}
where $$[m]^r:=m(m+1)\ldots (m+r-1)$$ $$ [m]_r:=m(m-1)\ldots (m-r+1) $$ are the upper and lower factorials. Now, we are ready to present our moment formula
\begin{align*} \label{moment1} 
C(2k,\ve,m) &= \frac{1}{1-\ve} \sum_{T \in \mathcal T_k}  \varphi(T,m)\mathcal M_{2k}(T)\psi(D(T), \ve) \\
C(2k+1, \ve, m) &= 0.
\end{align*}

\noindent  We would like to emphasize that  a sequence of constants may not correspond to the moments of any measure. Moreover, even if they do correspond to a measure, they might not uniquely define it. It was a  non-trivial task  to prove that  the constants  $C(k, \ve, m) $ defined above determine a unique measure.

\subsection{Remarks}   Our computation of the  constants $C(k,\ve,m)$ provides a quantitative explanation for the 
difference  of the spectrum from Wigner's semi-circle law.  In the computation with Erd\H{o}s-R\'{e}nyi graphs (which leads to the semi-circle law),  the moments are essentially the number of closed walks that can be performed on  trees of a given size \cite{Wigner}. For the PA graph, the story is  more complicated.  While the trees are still the family that we weight for the moment computation, one needs to consider: (1) trees of different sizes (2) distinct orderings on the trees (3) assignments  of weights which depend on the
trees' degree sequences. 
\\
\\
\noindent Numerically, one can 
compute the  approximate limit very quickly. For a given small $\ve$,  one can compute the first few moments and then use the inverse Fourier transform to obtain a distribution. The distribution matches the bulk (triangular) part of the distribution coming from the real graph closely; see Figure \ref{fitting}.
\\
\begin{figure}
	\begin{center}
		{\large }		\includegraphics[scale = 0.25]{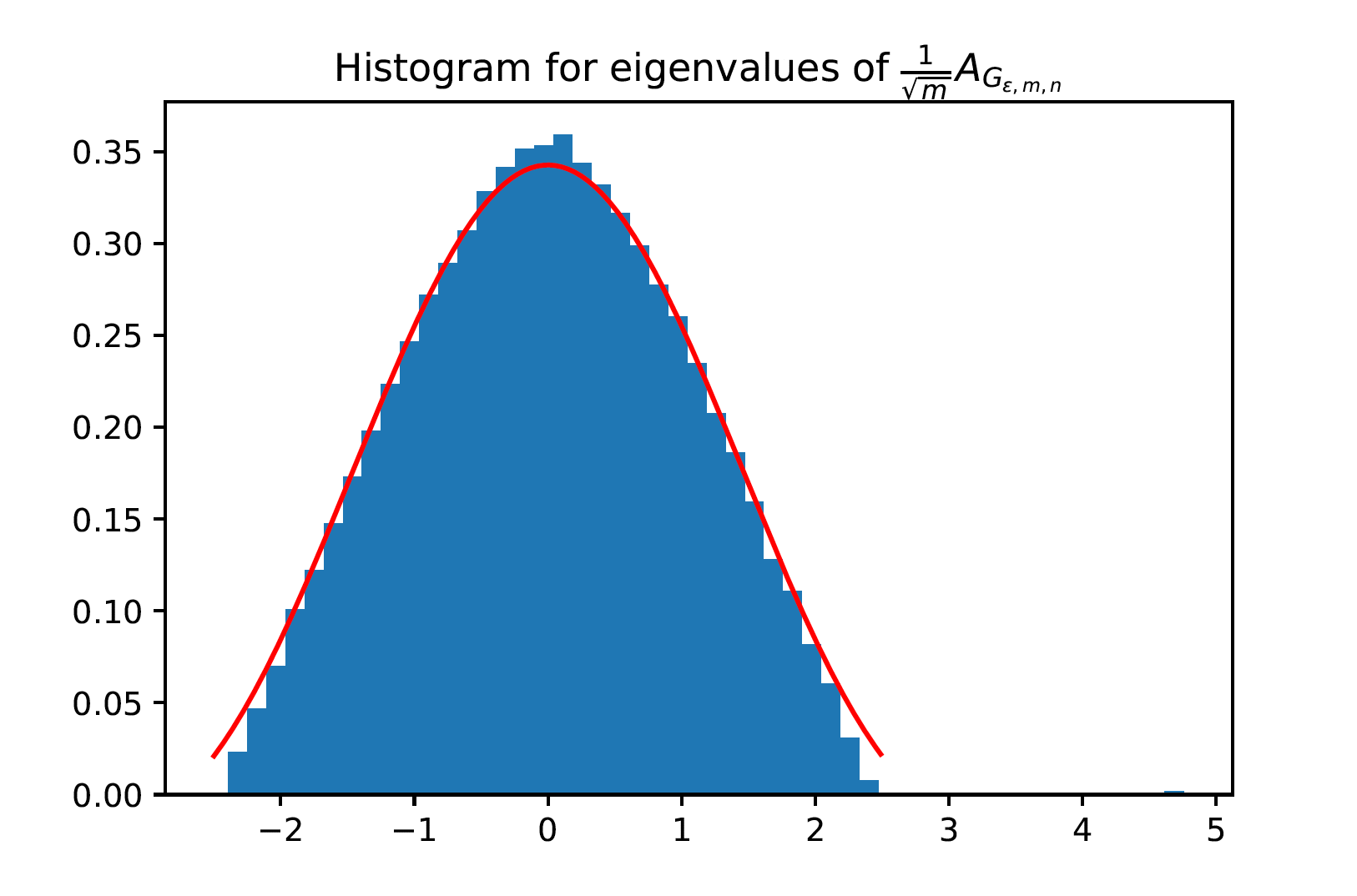}
	\end{center}
	\caption{In blue: A histogram of the eigenvalues of  $\frac{1}{\sqrt{m}}A_{G_{\ve,m,n}}$ with $\ve = 0.10$, $m= 15$ and $n=4000$. In red: We computed $C(k, 0.10, 15)$ for $k\leq 6$, and using the Inverse Fourier transform to obtain an approximation to the distribution.}\label{fitting}
\end{figure}

\subsection{ Localization of the leading eigenvectors} 

Our result concerning the eigenvectors 
is a rigorous 
theoretical justification of {\bf Observation 2}.   For a vector $v\in \mathbb R^n$, define
$$
\|v\|_\infty = \max_{j\in [n]} |v(j)|.
$$
\begin{theorem}\label{evec}Let $G_{m,n}$ be a PA random graph and  $K$ be a constant. With high probability\footnote{An event $\mathcal E$ is said to hold with high probability if $\BP[\mathcal E] = 1- o(1)$.}, 
	$$
	\|v_i\|_\infty = \frac{1}{\sqrt{2}} \pm o(1)
	$$for all $1\leq i\leq K$. Moreover, all the coordinates which do not realize the infinity norm are $o(1)$. 
\end{theorem} 
\begin{figure}
	\begin{center}
		\includegraphics[scale = 0.5]{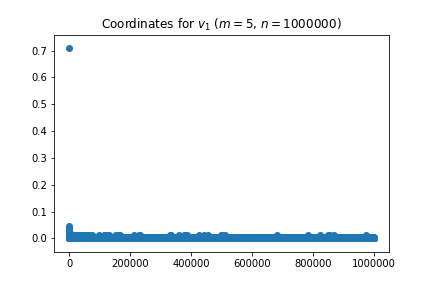}
	\end{center}
	\caption{Coordinates of $v_1$: We can see $\|v_1\|_\infty\approx \frac{1}{\sqrt{2}}$ and all the other coordinates are small.}\label{vector}
\end{figure}
\noindent We also obtained a new result regarding the power law of the edge eigenvalues. Let $\Delta_i, i=1,2, \dots$ be the degrees in decreasing order.

\begin{theorem}\label{edge1}Let $G_{m,n}$ be a PA random graph and $k= n^{1/25}$. Then 
	for $i=1,2,\ldots, k$, we have, whp, that  $$\lambda_i(G_{m,n}) = (1\pm o(1)) \sqrt{\Delta_i(G_{m,n})}.$$
\end{theorem}

\noindent This extends the main result of \cite{FFF} which yields the statement for $k=O(1)$.  It is an interesting open problem to find the threshold value for $k$. 
\\
\\
\noindent Our analysis offers a more satisfying explanation of {\bf Heuristics 3} mentioned in the introduction, which 
asserts that the power law of the leading eigenvalues is due 
to the localization of the corresponding eigenvectors. In fact, both properties (the power law and the localization) are consequences of the fact that the PA graph can be decomposed into a structural part and a random part. The structural part is a union of stars, whose leading eigenvalues satisfy a power law, and whose leading eigenvectors are localized. 
One can then prove that the impact of the random part on these eigenvalues and eigenvectors is negligible. Thus, the above properties still hold for the union of the two parts. 

\section{Notation and Combinatorial Tools}\label{notation section}
\noindent Fix $i<j$. To obtain an edge between these two points in $G_{m,n}$, by the definition of the graph process we must have had a point in $[(j-1)m + 1, jm]$ joining  a point in $[(i-1)m+1,im]$ in $G_{1,mn}$. There are  clearly $m^2$ possible pairs. A \textit{labeled} edge $(i,j)$ for us will be one specific choice among the $m^2$ choices. A \textit{labeled} graph will denote a collection of labeled edges. Since the out degree in $G_{1,mn}$ is exactly one, we only consider labeled graph with this property. Theorem \ref{bt} will give us an exact formula for the probability of a labeled graph to appear in $G_{m,n}$. 
\\
\\
\noindent Let $H$ be an ordered graph, and denote by $X_{m,n}(H)$ the number of subgraphs of $G_{m,n}$ isomorphic to $H$. That is, 
\begin{equation}\label{g count}
X_{m,n}(H) = \sum_{\substack{\tilde H \subset \cK_{m,n}\\ \tilde H \cong H}} \mathbb I_{\tilde H \subset G_{m,n}} .
\end{equation}Let $V$ be a subset of the vertices and  $G[V]$ be the subgraph of $G$ spanned by  $V$. We will be interested in counting the number of copies of $H$ is some fixed subset $V$. We extend the definition of graph counts as follows
\begin{equation}\label{g count v}
X_{m,n}(H,V) = \sum_{\substack{\tilde H \subset \cK_{m,n}[V]\\ \tilde H \cong H}} \mathbb I_{\tilde H \subset G_{m,n}[V]} .
\end{equation}Note that in particular, $X_{m,n}(H,[n]) = X_{m,n}(H)$. Throughout this paper we will count graphs which have bounded size. 
\\
\\
\noindent We are going to study the spectrum of a PA random graph. Since this is a multigraph, its adjacency matrix is not necessarily a $(0,1) $ matrix.

\begin{definition}Let $G$ be a multigraph on vertex set $[n]$. Define the adjacency matrix of $G$ to be the matrix $A_G$, where for $i\neq j$, $(A_G)_{ij}$ is the number of edges between $i$ and $j$, and for $i = j$, $(A_G)_{ii}$ is the number of loops at $i$.
\end{definition}
\noindent We will be using the following well known identity
\begin{lemma}Let $G$ be a multigraph, and let $A_G$ denote its adjacency matrix. Then
	 $$
	\Tr(A_G^k) = |\mathcal W_k(G)|.$$
\end{lemma}
\noindent Each walk in  $\mathcal W_k(G)$  defines a  connected subgraph of $G$,  obtained by taking the union of the edges which appear in the walk.
\\
\\
 \noindent Hence, if $G$ is a random graph which is always a subgraph of $\mathcal K$, then for any $W\in \mathcal W_k(\mathcal K)$, the event $``W\subset G"$ denotes the event that the subgraph given by $W$ appears as a subgraph of $G$. Hence, we can re-write the above as 
$$
\Tr(A_G^k) = \sum_{W\in \mathcal W_k(\mathcal K)} \mathbb I_{W\subset G} .
$$
For a connected subgraph $H$ of $\mathcal K$, let $\cM_k(H)$ denote the number of elements of $\cW_k(\mathcal K)$ which yield $H$ as a subgraph. Hence, if we denote by $\mathcal C(\mathcal K)$ the set of connected subgraphs of $\mathcal K$, we can re-write the above as 
$$
\Tr(A_G^k) = \sum_{H\in \mathcal C(\mathcal K)} \cM_k(H) \mathbb{I}_{H\subset G}
$$
Note that $\mathcal C(\mathcal K)$ is a set of labeled
 graphs, so in order to write it as graph counts, we need to decompose it even further. Let $\mathcal F(\mathcal K)$ denote the set of graphs which are isomorphic to a member of $\mathcal C(\mathcal K)$. Note that by our definition of isomorphism, we have that $\mathcal F(\mathcal K)$ is a set of ordered graphs. Furthermore, note that $\mathcal M_k(\cdot)$ was defined for labeled graphs, but since $H_1 \cong H_2$ implies $\mathcal M_k(H_1) = \mathcal M_k(H_2)$ then when we write $\mathcal M_k(H)$ for a   $H\in\mathcal F(\mathcal K)$ we refer to the obvious quantity. Hence, we arrive at the final form we will be using in this paper:
\begin{equation}\label{to use}
\Tr(A_G^k) = \sum_{H \in \mathcal F(\mathcal K)}\mathcal M_k(H) X(H, G) 
\end{equation}where $X(H,G)$ is the number of copies of $H$ in $G$. 
\\
\\
\noindent Consider the following example: Suppose we sampled the following graph $G_{2,2,} \in \mathcal G_{2,2}$:
\begin{center}
	\begin{tikzpicture}
	[acteur/.style={circle, fill=black!25, inner sep=0pt,minimum size=5pt,}] 
	\node (v1) at ( 0,0) [acteur][label=below:$1$]{};
	\node (v2) at ( 2,0)[acteur][label=below:$2$]{}; 
	\draw (v1) -- (v2);
	\draw[line width=0.5pt] (v1) to [out=45,in=135,looseness=30] (v1);
	\draw[line width=0.5pt] (v1) to [out=60,in=120,looseness=20] (v1);
	\draw[line width=0.5pt] (v2) to [out=60,in=120,looseness=20] (v2);
	\node (d)at (1,-0.1) [fill=white,draw=none][label=below:$G_{2,2}$]{};
	\end{tikzpicture}
\end{center}
Then, $$
A_{G_{2,2}} = \begin{bmatrix}
2 & 1 \\
1 & 1
\end{bmatrix} \qquad \text{ and} \qquad \Tr(A_{G_{2,2}}^2) = 7
$$
\noindent Now we carry out the computation as suggested by equation \eqref{to use}. The complete graph $\mathcal K_{2,2}$ is given by: 
\begin{center}
	\begin{tikzpicture}
	[acteur/.style={circle, fill=black!25, inner sep=0pt,minimum size=5pt,}] 
	\node (v1) at ( 0,0) [acteur][label=below:$1$]{};
	\node (v2) at ( 2,0)[acteur][label=below:$2$]{}; 
	\draw[line width=0.5pt] (v1) to [out=125,in=230,looseness=30] (v1);
	\draw[line width=0.5pt] (v1) to [out=125,in=230,looseness=20] (v1);
	\draw[line width=0.5pt] (v1) to [out=125,in=230,looseness=10] (v1);
	\draw[line width=0.5pt] (v2) to [out=55,in=320,looseness=30] (v2);
	\draw[line width=0.5pt] (v2) to [out=55,in=320,looseness=20] (v2);
	\draw[line width=0.5pt] (v2) to [out=55,in=320,looseness=10] (v2);
	\draw[line width=0.5pt] (v1) .. controls(1,0.15) .. (v2);
	\draw[line width=0.5pt] (v1) .. controls(1,0.30) .. (v2);
	\draw[line width=0.5pt] (v1) .. controls(1,-0.15) .. (v2);
	\draw[line width=0.5pt] (v1) .. controls(1,-0.30) .. (v2);
	\node (d)at (1,-0.35) [fill=white,draw=none][label=below:$\mathcal K_{2,2}$]{};
	\end{tikzpicture}
\end{center}
The only graphs $H\in \mathcal F(\mathcal K_{2,2})$ which have  $X(H, G_{2,2}) >0$ are the following: 
\begin{center}
	\begin{tikzpicture}
	[acteur/.style={circle, fill=black!25, inner sep=0pt,minimum size=5pt,}] 
	\node (v1) at ( 0,0) [acteur][label=below:$v_1$]{};
	\draw[line width=0.5pt] (v1) to [out=45,in=135,looseness=30] (v1);
	\node(e1) at(0,1.4)[fill=white,draw=none][label=below:$e_1$]{};
	\node (d)at (0,-0.5) [fill=white,draw=none][label=below:$H_1$]{};
	\end{tikzpicture}
	\begin{tikzpicture}
	[acteur/.style={circle, fill=black!25, inner sep=0pt,minimum size=5pt,}] 
	\node (v1) at ( 0,0) [acteur][label=below:$v_1$]{};
	\draw[line width=0.5pt] (v1) to [out=45,in=135,looseness=20] (v1);
	\draw[line width=0.5pt] (v1) to [out=45,in=135,looseness=30] (v1);
	\node(e1) at(0,1.4)[fill=white,draw=none][label=below:$e_1$]{};
	\node(e2) at(-0.35,0.1)[fill=white,draw=none][label=left:$e_2$]{};
	\node (d)at (0,-0.5) [fill=white,draw=none][label=below:$H_2$]{};
	\end{tikzpicture}
	\begin{tikzpicture}
	[acteur/.style={circle, fill=black!25, inner sep=0pt,minimum size=5pt,}] 
	\node (v1) at ( 0,0) [acteur][label=below:$v_1$]{};
	\node (v2) at (2,0)[acteur][label=below:$v_2$]{};
	\draw(v1) -- (v2);
	\node (d)at (1,-0.5) [fill=white,draw=none][label=below:$H_3$]{};
	\node(f) at(1,0.6)[fill=white,draw=none][label=below:$f_1$]{};
	\end{tikzpicture}
\end{center}
The closed walks of length two that may have formed each are: 
\begin{itemize}
	\item $H_1$: $(v_1,e_1,v_1,e_1,v_1)$. 
	\item $H_2$: $(v_1,e_1,v_1,e_2,v_1)$, and $(v_1,e_2,v_1,e_1,v_1)$. 
	\item $H_3$: $(v_1,f_1,v_2,f_1,v_1)$, and $(v_2, f_1,v_1,f_1,v_2)$.
\end{itemize}
Hence, $\mathcal M_2(H_1) = 1$, and $\mathcal M_2(H_2)=\mathcal M_2(H_3)= 2$. Computing the graph counts we get: $X(H_1, G_{2,2}) = 3$, $X(H_2, G_{2,2}) = 1$, $X(H_3, G_{2,2}) = 1$. Hence, 
$$
\sum_{H\in \mathcal F(\mathcal K_{2,2})} \mathcal M_2(H) X(H, G_{2,2})  = 1(3) + 2(1) + 2(1)= 7
= \Tr(A_{G_{2,2}}^2). $$
\begin{definition}Let $A_n$ be an $n\times n$ symmetric matrix. Denote by  $\{\lambda_i\}_{i=1}^n$ the eigenvalues of $A_n$. We define the empirical spectral distribution (ESD) to be the following measure:$$ \mu_{A_n} (x) := \frac{1}{n}\sum_{i = 1}^n \mathbb I_{\lambda_i = x}.$$
\end{definition}
\noindent Usually when it is required, we will denote the eigenvalues by $\lambda_i(A_n)$, but if the matrix $A_n$ is clear from the context, we will simply write $\lambda_i$. Moreover, it will always be assumed that the eigenvalues are labeled to be non-increasing, i.e., $\lambda_i\geq \lambda_{i+1}$. We are interested in studying the measures when the matrices $A_n$ follow the preferential attachment rule. We extend the definition of ESD to graphs. 
\begin{definition}Let $G$ be a graph on $n$ vertices, and let $A_G$ denote its adjacency matrix. Then the ESD of $G$ which we denote by $\mu_G$, is the ESD of the matrix $A_G$. 
\end{definition}
\noindent Note that when the graph $G$ is random, we will obtain a random measure. We are interested in studying the measures $\mu_{G_{m,n}}$ when $G_{m,n}\in \mathcal G_{m,n}$. In particular, we are interested in understanding their limiting behavior, which we formalize through the following notions. 
\begin{definition}\label{weak}Let $\{\mu_n\}_{n = 1}^{\infty}$ be a sequence of (deterministic) measures. We say that $\{\mu_n\}$ converges weakly to a measure $\mu$, if for every bounded, continuous function $f$ we have $$\int f d\mu_n \rightarrow \int fd\mu.$$
\end{definition}
\noindent We use the following convention: 
$$ 
\BE[f(\mu)] := \int f d\mu.
$$In the case when $f(x) = x^k$, we refer to $\BE[\mu^k]$ as the $k$-th moment of the measure. If $\mu$ is determined by its moments, then convergence of the moments guarantees weak convergence: 
\begin{lemma}\label{moments lemma}Let $\{\mu_n\}$ be a sequence of (deterministic) measures. If $\mu$ is uniquely defined by its moments, and for every $k\geq 0$ we have $$ \int x^k d\mu_n \rightarrow \int x^k d\mu$$ then $\{\mu_n\}$ converges weakly in probability to $\mu$. 
\end{lemma}
\begin{proof}The proof is a standard exercise, see Appendix. 
\end{proof}
\noindent In the case when $\mu$  is the ESD of a graph, we refer to $\BE[\mu_G^k]$ as the $k$-th spectral moment. Using the notation above we have,
\begin{align} 
\BE[\mu_G^k] 
&= \int x^k d\mu \nonumber\\
&= \frac{1}{|V(G)|} \sum_i \lambda_i^k\nonumber \\
&= \frac{1}{|V(G)|} \Tr(A_G^k) \nonumber\\ 
&= \frac{1}{|V(G)|} \sum_{H \in \mathcal F(G)}\mathcal M_k(H) X(H, G). \label{muk}
\end{align}
Note that when $\{\mu_n\}$ are random measures, then $\BE[f(\mu_n)]$ are random variables, so we need to specify the type of convergence one obtains in Definition \ref{weak}. The first and most natural definition deals with the behavior of a ``typical" measure $\mu_n$. 
\begin{definition}\label{exp}Let $A_n$ be a $n\times n$ random symmetric matrix, and let $\mu_{n}$ denote its ESD. We define the expected ESD, $\bar{\mu}_n$, via duality as follows: 
	$$ \int f d\bar\mu_n := \BE\left[\int fd\mu_n\right]$$Moreover, we say that $\{\mu_n\}$ converges in expectation to a fixed measure $\mu$ iff $\{\bar\mu_n\}$ converges weakly to $\mu$. 
\end{definition}
\noindent A stronger type of convergence is given by the following notion. 
\begin{definition}\label{weak p}Let $\{\mu_n\}$ be a sequence of random measures. We say that $\{\mu_n\}$ converges weakly in probability to a fixed measure $\mu$ if for any bounded continuous function $f$ and $\ve >0$ we have $$ \lim_{n\rightarrow \infty}\BP\left[\left |\int f d\mu_n - \int f d\mu\right| >\ve\right] = 0 .$$ 
\end{definition}
\noindent Just as before, if we know that $\mu$ is determined by its moments, then we can apply the moment method: 
\begin{lemma}\label{conv in moments}Let $\{\mu_n\}$ be a sequence of random measures, and let $\mu$ be a fixed measure which is uniquely defined by its moments. If $$ \int x^k d\mu_n \rightarrow \int x^k d\mu$$ in probability, then $\{\mu_n\}$ converges weakly in probability to $\mu$.
\end{lemma}
\begin{proof}This is a standard exercise, see Appendix. 
	
\end{proof}
\noindent As mentioned before, when we perform the moment computation for the truncated graph, we will need to consider the degree sequence of trees. The following lemma gives the number of trees with a given degree sequence (see Lemma 1 in \cite{Moon}): 
\begin{lemma}\label{moon}If we have positive integers $d_1,\ldots, d_{k+1}$ such that $\sum d_i = 2k$, then the number of trees $T$ on $k+1$ vertices with degree sequence equal to $(d_1,\ldots, d_{k+1})$ is given by $${k-1 \choose d_1-1, \ldots, d_{k+1}-1}.$$
\end{lemma}
\section{Tools from Linear Algebra}
\noindent In this section we present the linear algebra tools we will be using throughout the text. 
\begin{theorem}[Cauchy's interlacing theorem]\label{interlacing}Let $A$ be an $n\times n$ symmetric matrix, and let $B$ be an $m\times m$ (with $m<n$) principal submatrix (that is, $B$ is obtained by deleting columns and their corresponding rows). Suppose $A$ has eigenvalues $\lambda_n\leq \ldots \leq \lambda_1$ and $B$ has eigenvalues $\beta_m\leq \ldots \leq \beta_1$. Then, 
	$$
	\lambda_{m-k+1} \geq \beta_{m-k+1} \geq \lambda_{n-k+1}
	$$for $k = 1, \ldots, m$. 
	\end{theorem}
\begin{theorem}[Weyl's inequality]Let $A, B$ be $n\times n$ symmetric matrices with $\|B\| \leq \delta$. Let $A' = A+B$. Then: $$|\lambda_i(A)-\lambda_i(A')|\leq \delta$$
	for all $i$. (Again we assume that the eigenvalues are ordered increasingly.)
\end{theorem}
\begin{lemma} Let $A$ be an $n\times n $ matrix with non-negative entries. Then for any positive $\{c_i\}$, $i=1, \dots, n$, 
\begin{align}\label{trick} \|A\| \leq  \max_{1\leq i\leq n} \left\lbrace \frac{1}{c_i}\sum_{j=1}^n c_j a_{ij}\right\rbrace.\end{align}
\end{lemma}
\begin{proof}Let $C$ be a diagonal matrix with entries $\{c_1,c_2,\ldots, c_n\}$. Both $A$ and $C^{-1}AC$ have the same eigenvalues. Since $c_i>0$ and $A$ is non-negative, then $C^{-1}AC$ is also non-negative. For a non-negative matrix, the largest eigenvalue is upper bounded by the max row sum, which implies the result. 
\end{proof}

\begin{theorem}[Davis-Kahan, see Theorem 1 in \cite{YWS}]\label{dk}Let $A$ and $B$ be symmetric $n\times n$ matrices. Let $A = B + C$. Denote by $v_i(B), v_i(A)$ the eigenvectors corresponding to $\lambda_i(B)$ and $\lambda_i(A)$. Then, 

$$
\sin(\angle(v_i(B), v_i(A)) \leq \frac{\|C\|}{\min(|\lambda_{i-1}(A) -\lambda_i(B)|, |\lambda_{i+1}(A) - \lambda_i(B)|)}.
	$$
\end{theorem}

\section{Tools from Probability} 

\noindent In this section we state the probability bounds for graph counts that we will be using throughout the paper. Since we obtain $G_{m,n}$ by collapsing points in $G_{1,mn}$, it is natural to first study the probability of a labeled graph to appear in $G_{1,mn}$. Fortunately, this part was already done in \cite{Bollobas}:
\begin{theorem}[Theorem 13 in \cite{Bollobas}]\label{bt}Let $G_{1,n}$ be a random PA graph. Let $S$ be a labeled graph with out-degree of every vertex at most $1$. Define $C_S(t)$ to be the number of edges $(i,j)$ in $S$ such that $i\leq t$ and $t\leq j$. Then: 
	$$
	\BP[S] = \left(\prod_{i \in V^-(S)} d_{in}(i , S)\right) \left( \prod_{i \in V^+(S)} \frac{1}{2i-1}\right) \left( \prod_{i \not \in V^+(H)} \left( 1 + \frac{C_S(i)}{2i-1}\right)\right)
	$$Furthermore, 
	\begin{equation}\label{prob}
	\BP[S] =  \left(\prod_{i \in V^-(S)} d_{in}(i , S)!\right) \left(\prod_{(i,j) \in E(S)} \frac{1}{2\sqrt{ij}}\right) \exp\left( O \left( \sum_{ i \in V(S)} C_S(i)^2 / i\right)\right).
	\end{equation}
\end{theorem}

\noindent We will be using only the second statement of the above theorem in this paper. Note that in particular when $S=(u,v)$, we obtain $\BP[S] \approx 1/2\sqrt{uv}$. Hence, if $S$ has more than one edge, then the second factor is exactly what one would obtain if the edges appeared independently. The third factor is a correction factor which tends to one. The first factor captures the correlation among the edges in the preferential attachment model. For instance if we have $a<b<c$, then we would obtain $\BP[(a,b),(a,c)] \approx 2 \BP[(a,b)] \BP[(a,c)]$ (since the in-degree of $a$ is two). 
\\
\\
Given an ordered graph $H$, and a subset of vertices $V$, we could have multiple copies of $H$ in the vertex set $V$. For example: 
\begin{example} Let $H$ be a triangle, then the total number of copies of $H$ that can occur in a given triplet of vertices $\{a,b,c\}$ is $m^5(m-1)$. To see this, note that to obtain a triangle in $G_{m,n}$ we must have points $a_1,a_2\in [(a-1)m+1,am]$, $b_1,b_2\in [(b-1)m+1,bm]$ and $c_1,c_2\in [(c-1)m+1,cm]$ such that during the construction of $G_{1,mn}$ we joined points $(a_1,b_1), (a_2,c_1)$ and $(b_2,c_2)$. We have $m^2$ choices to choose $a_1$ and $a_2$ since they could be the same point, and similarly for $b_1,b_2$, but since the out-degree in $G_{1,mn}$ is $1$, then we must have $c_1\neq c_2$ giving $m(m-1)$ options.
\end{example}
\noindent Another ingredient one needs to keep in mind when performing graph counts is that all the $m^5(m-1)$ triangles are not ``weighted" the same. As we mentioned before, we see that when $a_1=a_2$ we obtain an extra factor of two on the probability. Thus, to perform a graph count we need to keep in mind the number of ways the graph could have been obtained from the uncollapsed graph $G_{1,mn}$ and then assign it a weight according to the first factor in \eqref{prob}. As an illustration on how the count should be carried out, consider the following: 
\\
\\
\textbf{Example: }We want to compute the expected number of paths of length two in a random $G_{m,n}$ graph. There are a total of three non-isomorphic paths of length two (see Definition \ref{ordered}): 
\newline
\begin{center}
	\begin{tikzpicture}
	[acteur/.style={circle, fill=black!25, inner sep=0pt,minimum size=5pt,}] 
	\node (v1) at ( 0,0) [acteur][label=below:$v_1$]{};
	\node (v2) at ( 2,0)[acteur][label=below:$v_2$]{}; 
	\node (v3) at ( 4,0) [acteur][label=below:$v_3$]{}; 
	\draw (v1) -- (v2);
	\draw[line width=0.5pt] (v1) .. controls(2,0.5) .. (v3);
	\node (d)at (2,-0.5) [fill=white,draw=none][label=below:$H_1$]{};
	\end{tikzpicture} \qquad
	\begin{tikzpicture}
	[acteur/.style={circle, fill=black!25, inner sep=0pt,minimum size=5pt,}] 
	\node (v1) at ( 0,0) [acteur][label=below:$v_1$]{};
	\node (v2) at ( 2,0)[acteur][label=below:$v_2$]{}; 
	\node (v3) at ( 4,0) [acteur][label=below:$v_3$]{}; 
	\draw (v1) -- (v2);
	\draw(v2) -- (v3);
	\node (d)at (2,-0.5) [fill=white,draw=none][label=below:$H_2$]{};
	\end{tikzpicture} \qquad
	\begin{tikzpicture}
	[acteur/.style={circle, fill=black!25, inner sep=0pt,minimum size=5pt,}] 
	\node (v1) at ( 0,0) [acteur][label=below:$v_1$]{};
	\node (v2) at ( 2,0)[acteur][label=below:$v_2$]{}; 
	\node (v3) at ( 4,0) [acteur][label=below:$v_3$]{}; 
	\draw (v3) -- (v2);
	\draw[line width=0.5pt] (v1) .. controls(2,0.5) .. (v3);
	\node (d)at (2,-0.5) [fill=white,draw=none][label=below:$H_3$]{};
	\end{tikzpicture} 
\end{center}
We will count each one separately. Recall the notation from equations \eqref{g count} and \eqref{g count v}. For $H_1$, we have: 
\begin{align*}
\BE[X_{m,n}(H_1)] = \sum_{1\leq a< b< c\leq n} \BE[X_{m,n}(H_1, \{a,b,c\})]
\end{align*}
Let $A =[(a-1)m +1 , am]$, $B = [(b-1)m+1, bm]$ and $C= [(c-1)m+1,cm]$. Then all the possible contributions to $X_{m,n}(H_1,\{a,b,c\})$ come from having an edge $(a_1,b_1)$ from $B$ to $A$ and an edge $(a_2,c_1)$ from $C$ to $A$ (in the uncollapsed graph $G_{1,mn}$). If $a_1 = a_2$, then by Theorem \ref{bt}, the probability of $(a_1,b_1), (a_1,c_1)$ is given by: 
$$
2 \cdot \frac{1}{4 a_1\sqrt{b_1c_1}}\exp\left(O \left( \frac{1}{a_1} + \frac{4}{b_1} +\frac{1}{c_1}\right)\right)
$$which can be rewritten as: $$
2 \cdot \frac{1}{4m^2a\sqrt{bc}} \left(1 \pm O\left(\frac{1}{a}\right) \right)
$$There are $m$ choices for $a_1$, $b_1$, and $c_1$, so we obtain an extra factor of $m^3$ to arrive at: 
\begin{equation}\label{e1}
\frac{2m}{4a\sqrt{bc}}  \left(1 \pm O\left(\frac{1}{a}\right) \right)
\end{equation}
If however $a_1$ and $a_2$ are distinct, then we do not have a factor of $2$, and now there are $m(m-1)$ ways to choose them from $A$, so this case yields a contribution of 
\begin{equation}\label{e2}
\frac{m^3(m-1)}{4m^2 a\sqrt{bc}}  \left(1 \pm O\left(\frac{1}{a}\right) \right)
\end{equation}
Hence, by adding equations \eqref{e1} and \eqref{e2} we get:
\begin{equation}\label{e3}
\BE[X_{m,n}(H_1, \{a,b,c\})] = \frac{m(m+1)}{4a\sqrt{bc}}  \left(1 \pm O\left(\frac{1}{a}\right) \right)
\end{equation}Taking the sum over all $1\leq a<b<c\leq n$ gives: 
\begin{equation}\label{example}
\BE[X_{m,n}(H_1)] = (1\pm o(1)) \frac{m(m+1) n\log n}{2}
\end{equation}Similarly we can compute the contributions of $\BE[X_{m,n}(H_2)]$ and $\BE[X_{m,n}(H_3)]$, but they both turn out to be $O(n)$. Hence the expected number of paths of length two is given by: 
\begin{align}\label{p2}
(1\pm o(1)) \frac{m(m+1) n\log n}{2}.
\end{align}

\PRLsep

\noindent From the above example we see the importance of Definition \ref{ordered}. Since different orderings give completely different counts, it is natural to treat them separately. When we are concerned about the magnitude of the count, and do not care about the leading constant, the following theorem will give us a quick way to deal with such situations. 
\begin{theorem}\label{mag}Let $H$ be an ordered graph. Then: 
	\begin{equation}\label{count}
	\BE[X_{m,n}(H)]= O\left( n^{f(H)/2} \log^{g(H)}n \right)
	\end{equation}where $f(H)$ denotes the number of vertices of $H$ with degree one, and $g(H)$ denotes the number of vertices of $H$ of degree $2$. Moreover, if the ordering is such that $d(v_i, H) \geq d(v_{i+1}, H)$ then one can replace $O(\cdot)$ by $\Theta(\cdot)$. 
\end{theorem}
\begin{proof}Let $t$ denote the number of vertices of $H$. Then:
	$$
	\BE[X_{m,n}(H)] = \sum_{V\subset[n] : |V|= t} \BE[X_{m,n}(H,V)].
	$$
	For $ V = \{x_1<\ldots< x_t\}$, we have:
	\begin{equation}\label{bound}
	\BE[X_{m,n}(H,V)] = \Theta \left(\prod_{i=1}^t \frac{1}{x_i^{d(v_i,H)/2}}\right).
	\end{equation}To see this, let $\mathcal S$ denote the set of labeled graphs in $G_{1,mn}$ which collapse to a copy of $H$ on the vertex set $V$. Then: 
	$$
	\BE[X_{m,n}(H,V)] = \sum_{S\in \mathcal S} \BP[S\subset G_{1,mn}]
	$$By Theorem \ref{bt} equation \eqref{prob} we see that for each fixed $S$ we have: 
	\[
	\BP[S \subset G_{1,mn}] = \Theta(1) \left(\prod_{(i,j)\in E(S)}\frac{1}{2 \sqrt{ij}}\right) 
	\]By taking sum over all possible $S\in\mathcal S$ gives equation \eqref{bound}. Since $m$, and $|E(H)|$ are constants we obtain: 
	\begin{align*}
	\BE [X_{m,n}(H)] &= \sum_{1\leq x_1 < \ldots < x_t\leq n} \BE\left[ X_{m,n}(H, \{x_1,\ldots, x_t\})\right] \\
	&=  \sum_{1\leq x_1 < \ldots < x_t\leq n} \Theta \left( \prod_{i=1}^t \frac{1}{x_i^{d(v_i,H)/2}}\right) \\ 
	&= O\left(\prod_{i=1}^t \left(\sum_{1\leq x_i \leq n} \frac{1}{x_i^{d(v_i,H)/2}} \right)\right) \\
	&= O(n^{f(H)/2}\log^{g(H)}).
	\end{align*}
	For the moreover part of the theorem, see the Appendix. 
\end{proof}
\noindent As we will see later in the paper, we will sometimes be interested in counting ordered graphs where the smallest vertex is large. That is, given an ordered graph $H$, and $M = M(n)$ a function which tends to infinity, then we can consider the copies of $H$ which appear on $[M,n]$. Note that in some cases, this quantity will be asymptotically equal to the graph count in the whole graph.  To perform these counts more accurately, we present Lemma \ref{to-use}: 
\begin{lemma}\label{to-use}Let $H$ be an ordered loopless graph on $t$ vertices. Let $V\subset [M,n]$ be a subset of $t$ vertices $\{x_1<\ldots <x_t\}$. Then: 
	\begin{equation}\label{prob 2} 
	\BE[X_{m,n}(H,V)] = \left( 1 \pm O\left(\frac{1}{M}\right) \right) \frac{\varphi(H,m)}{(2m)^{|E(H)|}} \prod_{i=1}^t \frac{1}{x_i^{d(v_i,H)/2}}
	\end{equation}where we recall the definition of $\varphi$: $$
	\varphi(H, m) = \left( \prod_{v\in V^{-}(H)} [m]^{d_{in}(v,H)}\right) \left(\prod_{v\in V^{+}(H)} [m]_{d_{out}(v, H)}\right).
	$$
\end{lemma}
\noindent Note that the above lemma takes care of a couple of trivial cases as well: If a graph $H$ has out-degree more than $m$, then we know it can never appear in our process, and the lower factorials do indeed give zero if any out degree is more than $m$. Furthermore, it saves a lot of work when computing graph counts: If we revisit the situation where $H= \{(v_1,v_2),(v_1,v_2)\}$, then note that $\varphi(H, m) = m^3 (m+1)$, so we obtain equation \eqref{e3} immediately. 
\begin{proof}[Proof of lemma \ref{to-use}]Let $\mathcal S(H,V)$ be the set of graphs in $G_{1,mn}$ such that they form a copy of $H$ in the vertex set $V$ after being collapsed. We want to compute the probability $S\subset G_{1,mn}$ for $S\in \mathcal S(H,V)$, by Theorem \ref{bt} this is given by : 
	$$
	\BP[S] = \left(\prod_{i \in V^-(S)} d_{in}(i , S)!\right) \left(\prod_{(i,j) \in E(S)} \frac{1}{2\sqrt{ij}}\right) \exp\left( O \left( \sum_{ i \in V(S)} C_S(i)^2 / i\right)\right)
	$$the last factor will be always $(1\pm O(1/M))$. The second factor we can write as: 
	$$
	\frac{1}{(2m)^{|E(H)|}}\prod_{i=1}^t \frac{1}{x_i^{d(v_i,H)/2}}\left( 1 \pm O\left(\frac{1}{M}\right)\right)
	$$Thus, 
	$$
	\BE[X_{m,n}(H,V)] = \left( 1 \pm O\left(\frac{1}{M}\right)\right)\frac{1}{(2m)^{|E(H)|}}\prod_{i=1}^t \frac{1}{x_i^{d(v_i,H)/2}}\sum_{S\in \mathcal S(H,V)} \left(\prod_{i \in V^-(S)} d_{in}(i , S)!\right).
	$$We focus only on the last factor. Let $r$ be the number of edges of $H$, and denote these edges by $e_j= \{(a_j, b_j)\}$. Let $X_j$ denote the subset of $\{a_1,\ldots,a_r, b_1\ldots, b_r\}$ which corresponds to vertex $x_j$, in particular note that $|X_j| = d(v_j, H)$. We will further split each $X_j$ into two types: In-type and out-type which will correspond to whether they were an in-vertex or an out vertex, and call these sets $I(X_j)$ and $O(X_j)$ respectively. For example, if we consider $H$ to be the following ordered graph with the given labeling:
	\newline
	\begin{center}
		\begin{tikzpicture}
		[acteur/.style={circle, fill=black!25, inner sep=0pt,minimum size=5pt,}] 
		\node (v1) at ( 0,0) [acteur][label=below:$v_1$]{};
		\node (v2) at ( 2,0)[acteur][label=below:$v_2$]{}; 
		\node (v3) at ( 4,0) [acteur][label=below:$v_3$]{}; 
		\node (v4) at (6,0)[acteur][label=below:$v_4$]{};
		\draw (v1) -- (v2);
		\draw (v2) -- (v3);
		\draw (v3) -- (v4);
		\draw[line width=0.5pt] (v1) .. controls(2,0.5) .. (v3);
		\draw[line width=0.5pt] (v2) .. controls(4,0.5) .. (v4);
		\node (d)at (3,-0.5) [fill=white,draw=none][label=below:$H$]{};
		\node (e1)at (1,0.14) [fill=white,draw=none][label=below:$e_1$]{};
		\node (e2)at (2,0.85)[fill=white,draw=none][label=below:$e_2$]{};
		\node (e3)at (3,0.14) [fill=white,draw=none][label=below:$e_3$]{};
		\node (e4)at (4,0.85)[fill=white,draw=none][label=below:$e_4$]{};
		\node (e5)at (5,0.14) [fill=white,draw=none][label=below:$e_5$]{};
		\end{tikzpicture}
	\end{center}
	Then
	\begin{itemize}
		\item $X_1 = \{a_1,a_2\}$, $I(X_1) = \{a_1,a_2\}$ and $O(X_1) = \emptyset$. 
		\item $X_2 = \{b_1,a_3,a_4\}$, $I(X_2)=\{a_3,a_4\}$ and $O(X_2)=\{b_2\}$. 
		\item $X_3 = \{b_3,b_2,a_5\}$, $I(X_3) =\{a_5\}$, and $O(X_3) = \{b_2,b_3\}$. 
		\item $X_4 = \{b_4,b_5\}$, $I(X_4) =\emptyset$, and $O(X_4) = \{b_4,b_5\}$.
	\end{itemize}
	Note that we can think of $S\in \mathcal S(H,V)$ as maps $\sigma$ such that for a point $u\in X_j$ we have $\sigma(u) \in [(x_j -1 ) m + 1, x_j m]$. Furthermore, given any mapping $\sigma$, which takes $\sigma: X_j \rightarrow [(x_j -1 ) m + 1, x_j m]$ which satisfies $|\sigma(O(X_j))| = d_{out}(v_j, H)$ (this is  the condition that the out vertices have degree one), then we can construct a graph. 
	\\
	\\
	For a map $\sigma$, we can define a function $w(\sigma)$ which equals $\prod_{j \in V^-(S)} d_{in}(j , S)!$ where $S$ is the graph corresponding to $\sigma$. Let $\Omega$ be the set of all such mappings. Then, the factor we are interested in equals $$\sum_{\sigma \in \Omega} w(\sigma)$$Note that every $\sigma$ is uniquely defined by its action on $X_j$. Hence, we can break them up as $\sigma = \sigma_1 \times \ldots \times\sigma_t$. Let $\Omega_j$ be the set of all mappings $\sigma_j$ from $X_j$ to $ [(x_j -1 ) m + 1, x_j m]$ which satisfies $|\sigma(O(X_j))| = d_{out}(v_j, H)$. Given one $\sigma_j$, we can define $$ w(\sigma_j) = \prod_{u \in Im(\sigma_j)} |\sigma_j^{-1}(u) \cap I(X_j)|!$$Note that $$w(\sigma) = \prod_{i=1}^t w(\sigma_j)$$ where $\sigma_j$ is the restriction of $\sigma$ to $X_j$. Thus, we obtain: 
	\begin{align*}
	\sum_{S\in \mathcal S(H,V)} \prod_{j \in V^-(S)} d_{in}(j , S)! &= \sum_{(\sigma_1,\ldots, \sigma_t)\in \Omega_1\times\ldots \times \Omega_t} \prod_{j=1}^t w(\sigma_j)\\
	&= \prod_{j=1}^t \sum_{\sigma_j \in \Omega_j} w(\sigma_j)
	\end{align*}
	Fix $j$, and for the remainder of the proof write $d_j$ instead of $d_{in}(v_j,H)$. First of all note that the only restriction on where to map $O(X_j)$ is that they must map to distinct points. This will gives us a factor of $[m]_{d_{out}(v_j,H)}$. To map $I(X_j)$ and weight the maps accordingly, we have:
	\begin{itemize}
		\item If we choose the images to be $d_j$ distinct points, then we have $[m]_{d_j}$ many maps with  $w(\sigma_j) = 1$.
		\item If we choose the images to be $d_j-1 $ distinct points, then we have ${d_j-1 \choose 2}[m]_{d_j-1}$ many maps with $w(\sigma_j)=2$. \\
		$\vdots$
		\item If we choose the images to be $1$ distinct point, then we have $[m]_1$ many maps with $w(\sigma_j) = d_j!$. 
	\end{itemize}In general, the contribution of adding over maps $\sigma$ with exactly $k$ points in the image is given by: 
	$$
	\frac{d_j ! {d_j -1 \choose k-1} }{k!} [m]_k.
	$$Indeed, First we permute the points. Then we partition them into $k$ non-empty sets, and for the first partition we have $m$ ways to select its image, for the second partition we have $m-1$ ways to select its image and so on. Lastly, we divide by $k!$ to account for the double counting given by the ways we can re-arrange the partition amongst themselves. The key point here is that we do not further divide by the number of ways we can re-arrange the inside of these partitions. This will make each mapping $\sigma_j$ be counted a total of $\prod_{u \in Im(\sigma_j)} |\sigma_j^{-1}(u) \cap I(X_j)|!$ many times (which is exactly the factor we wanted to weight it by). Hence,
	\begin{align*}
	\sum_{\sigma_j \in \Omega_j} w(\sigma_j) &= [m]_{d_{out}(v_j,H)}\sum_{k=1}^{d_j } \frac{d_j ! {d_j -1 \choose k-1} }{k!} [m]_k\\
	&= [m]_{d_{out}(v_j,H)}d_{j} ! \sum_{k=1}^{d_{j }} {d_j-1\choose k-1}{m \choose k}\\
	&= [m]_{d_{out}(v_j,H)}d_{j} !\sum_{k=0}^{d_{j}-1 } {d_{j} -1 \choose d_{j} -1-k}{m\choose k+1}\\
	&= [m]_{d_{out}(v_j,H)}d_{j} ! {m+ d_j -1 \choose d_j} \\
	&= [m]_{d_{out}(v_j,H)}[m]^{d_j}\\
	&=[m]_{d_{out}(v_j,H)}[m]^{d_{in}(v_j, H)}
	\end{align*}That is, 
	\begin{align*}
	\sum_{S\in \mathcal S(H,V)} \prod_{j \in V^-(S)} d_{in}(j , S)! &=  \prod_{j=1}^t \sum_{\sigma_j \in \Omega_j} w(\sigma_j)\\
	&= \prod_{j=1}^t[m]_{d_{out}(v_j,H)}[m]^{d_{in}(v_j, H)}\\
	&= \varphi(H, m)
	\end{align*}as defined in the theorem. This concludes the proof.
\end{proof}
\noindent Lastly, we will be using the following lemma:
\begin{lemma}[Lemma 3 in \cite{BR}]\label{neg corr}Let $H_1$ and $H_2$ be labeled graphs (that is, some specific graphs of $\mathcal K_{m,n}$), which are vertex disjoint. Then: 
	$$
	\BP[H_1 \cup H_2] \leq \BP[H_1] \BP[H_2].
	$$That is, disjoint graphs are negatively correlated.
\end{lemma}
\noindent The above lemma is of course no surprise. When we condition on one of the graphs appearing, this reduces the expected degree of the other one (since we know that some of the edges ``missed" connections), and thus they are negatively correlated. 
\\
\\
\noindent Recall from above that we will be describing the approximate limit of the spectrum via its moments. There are two natural  issues with this strategy: Is there a distribution whose moments match the given constants $C(k,\ve,m)$, and if so, is it  unique?
\\
\\
The first question is known in the literature as the Hamburger moment problem: Given a sequence $\{m_k\}_{k\geq 0}$, we say that it is solvable if there is a positive Borel measure $\mu$ on the real line such that the moments of $\mu$ are given by $\{m_k\}$. If the sequence is indeed solvable, the question of whether the measure is unique is known as $M$-determinacy (short for moment-determinacy). 
\begin{theorem}[Hamburger Theorem] A sequence $\{m_k\}_{k\geq 0}$ is solvable if and only if for an arbitrary sequence $\{c_j\}_{j\geq 0}$ of complex numbers with finite support, one has: 
	\begin{equation}\label{hamb}
	\sum_{r,s\geq 0} m_{r+s} c_r \bar{c}_s \geq 0.
	\end{equation}
\end{theorem}
\begin{theorem}[Carleman's condition]Let $\mu$ be a measure on $\mathbb R$ such that all the moments $$m_k = \int x^k d\mu$$ are finite. If  
	\begin{equation}\label{condition}
	\limsup_{k \rightarrow \infty} \frac{m_{2k}^{1/2k}}{2k} < \infty
	\end{equation}
	then $\mu$ is the only measure on $\mathbb R$ with $\{m_k\}$ as its sequence of moments. That is, $\{m_k\}$ is $M$-determinate. 
\end{theorem}\label{Moments}

\section{ Spectral moments}

\noindent In this section we find the spectral moments of a random PA graph. Note that in \cite{PJ} they conjectured that in order to carry out the moment computation for a random preferential attachment graph, it suffices to consider the edges as if they appeared independently (see their Conjecture 1). Although their model is slightly different than the one used here, we can see that the conjecture does not hold since the preferential attachment rule by which we add edges gives us the extra factor in equation \eqref{prob}. 

\begin{theorem}\label{moments thm}Let $G_{m,n}$ be a PA random graph, and let $\mu_{m,n}$ denote its ESD. Let $\bar\mu_{m,n}$ be the expected ESD. Then the first spectral moments are: 
	\begin{align*} 
	\BE[\bar\mu_{m,n}^0] &= 1 \\
	\BE[\bar\mu_{m,n}^1] &= \Theta\left(\frac{\log n}{n}\right) \\ 
	\BE[\bar\mu_{m,n}^2] &= (1+ o(1)) 2m \\
	\BE[\bar\mu_{m,n}^3] &= \Theta\left(1/\sqrt{n}\right) \\ 
	\BE[\bar\mu_{m,n}^4] &= (1+ o(1)) 2m(m+1) \log n
	\end{align*}and for $k \geq 3$ we have: 
	\begin{align}
	\BE[\bar\mu_{m,n}^{2k}] &= \Theta(n^{k/2 - 1})\label{evens}\\
	\BE[\bar\mu_{m,n}^{2k-1}] &= \Theta(n^{k/2 - 3/2}). \label{odds}
	\end{align}
\end{theorem}
\begin{remark}Note that in particular we obtain the same first, second and third moments, but the 4th-moment is different in \cite{PJ}. The extra factor in equation \eqref{prob} which they do not use changes the leading coefficient from $m^2$ to $m(m+1)$. 
\end{remark}
\begin{proof}[Proof of Theorem \ref{moments thm}]For short hand, we are going to write $M_k$ instead of $\BE[\bar\mu_{m,n}^k]$. Using equation \eqref{muk} we have:
	\begin{align*}
	M_k
	&= \frac{1}{n}\BE\left[\sum_{H \in \mathcal F(\cK_{m,n})}\mathcal M_k(H) X_{m,n}(H)\right]\\
	&= \frac{1}{n} \sum_{H\in \mathcal F(\cK_{m,n})} \mathcal M_k(H) \BE[X_{m,n}(H)] .
	\end{align*}
	We mention a few words on moments $M_k$ for $k\leq 4$ and then handle the remaining cases. We have to treat them separately since they are corner cases. 
	\begin{itemize}
		\item $M_0$ follows from the definition of a pdf. 
		\item $M_1$ comes from loops. By Theorem \ref{mag} the expected number of loops is $\Theta(\log n)$. 
		\item $M_2$ comes from edges and loops. The contribution from loops is again $\Theta(\log n)$, and the number of edges is $mn$. Since we can transverse any edge in two ways we get $M_2 = (1+o(1))2m$. 
		\item $M_3$ will be dealt with in the general case below. 
		\item $M_4$ the contributions come from: A collection of loops around a vertex (contribution of $\Theta(\log n)$), a collection of loops and an edge (contribution of $\Theta(\sqrt{n})$ by using Theorem \ref{mag} and placing the vertex with the loop before the leaf), an edge which gives a contribution of $\Theta(n)$. The only remaining contribution is from paths of length two which we did in our graph count example. If $H$ is a path of length two, then $\mathcal M_4(H)$ equals $4$. Hence, $$M_4 = (1\pm o(1))\left(4 \cdot \frac{m(m+1)\log n}{2}\right).$$
	\end{itemize}
	For the general case we split the proof into the odd case and the even case. 
	\\
	\\
	\textbf{Even moments}: We aim to compute $M_{2k}$. By equation \eqref{muk}: 
	\begin{align*} 
	M_{2k} = \frac{1}{n} \sum_{H\in \mathcal F(\mathcal K_{m,n})} \mathcal M_{2k}(H) \BE[X_{m,n}(H)].
	\end{align*}
	By equation \eqref{mag} we see that the contribution of an ordered graph $H$ is $O(n^{f(H)/2 -1} \log^{g(H)} n)$. Since $H$ comes from a closed walk of length $2k$, we must have $f(H) \leq k$ (to reach each leaf we must walk their corresponding edge back and forth), with equality if and only if we have a star on $k$ edges. This yields the upper bound of $O(n^{k/2 - 1})$. Now note that this bound is indeed achievable by Theorem \ref{mag} if we let $H$ be the following ordered graph:
	\begin{center}
		\begin{tikzpicture}
		[acteur/.style={circle, fill=black!25, inner sep=0pt,minimum size=5pt,}] 
		\node (v1) at ( 0,0) [acteur][label=below:$v_1$]{};
		\node (v2) at ( 2,0)[acteur][label=below:$v_2$]{}; 
		\node (v3) at ( 4,0) [acteur][label=below:$v_3$]{}; 
		\node (d)at (6,0) [fill=white,draw=none][label=below:$\ldots$]{};
		\node (v4) at ( 8,0) [acteur][label=below:$v_{k+1}$]{};
		\draw (v1) -- (v2);
		\draw[line width=0.5pt] (v1) .. controls(2,0.5) .. (v3);
		\draw[line width=0.5pt] (v1) .. controls(4,1.5) .. (v4);
		\end{tikzpicture} 
	\end{center}
	\textbf{Odd moments:} We want to compute $M_{2k-1}$. Just as above, note that for any $H$ which yields a closed walk of length $2k-1$ we must have $f(H)\leq k-1$. To see this, note that any closed walk must have a cycle, which uses at least one edge (in the case of a loop), and then we can have at most $k-1$ vertices of degree one. Note that we can only have equality when we have a start and a loop around the high degree vertex. This yields an upper bound of $O(n^{(k-1)/2 -1})$. In the case when $H$ is the following ordered graph:
	\begin{center}
		\begin{tikzpicture}
		[acteur/.style={circle, fill=black!25, inner sep=0pt,minimum size=5pt,}] 
		\node (v1) at ( 0,0) [acteur][label=below:$v_1$]{};
		\node (v2) at ( 2,0)[acteur][label=below:$v_2$]{}; 
		\node (v3) at ( 4,0) [acteur][label=below:$v_3$]{}; 
		\node (d)at (6,0) [fill=white,draw=none][label=below:$\ldots$]{};
		\node (v4) at ( 8,0) [acteur][label=below:$v_{k+1}$]{};
		\draw (v1) -- (v2);
		\draw[line width=0.5pt] (v1) .. controls(2,0.5) .. (v3);
		\draw[line width=0.5pt] (v1) .. controls(4,1.5) .. (v4);
		\draw[line width=0.5pt] (v1) to [out=45,in=135,looseness=40] (v1);
		\end{tikzpicture} 
	\end{center}
	we obtain a matching lower bound, which implies the result. 
\end{proof}
\noindent As we pointed out during the introduction, in order to further study the spectrum of the PA graph, we need to break our analysis into two: The bulk of the spectrum and the edge eigenvalues.

\section{Spectrum of the truncated graphs and Proof of Theorem \ref{theorem:limit}}\label{section truncated}

\noindent In this section we will prove Theorem \ref{theorem:limit}. As we will see in a later section, the largest eigenvalues will arise from the stars with large degree (which correspond to vertices appearing early in our graph). Hence, to exclude these large eigenvalues from our study, we will consider the graph obtained by deleting the first vertices: 
\begin{definition}Let $\ve>0$ be fixed. An $\ve$-truncated PA random graph is a random graph obtained by sampling $G_{m,n}\in \mathcal G_{m,n}$, and then deleting the first $\ve n$ vertices to obtain $G_{\ve,m,n}$. We will denote this space of random graphs by $\mathcal G_{\ve,m,n}$. Note that these will be a graphs on $n-\ve n$ vertices and we will refer to its vertices by its original label. That is, $V(G_{\ve,m,n}) = [\ve n, n]$. 
\end{definition}
\noindent We want to study the limiting behavior of the random measures $\mu_{G_{\ve,m,n}}$ which henceforth we will write as $\mu_{\ve,m,n}$ to shorten the notation. Note that the first part of Theorem \ref{theorem:limit} will be given by Cauchy's interlacing theorem (Theorem \ref{interlacing}). Namely, 
$$
dist(\mu_{m,n}, \mu_{\ve,m,n})\leq \ve.
$$
For the second part of the theorem, we define $\mu_{\infty, \ve, m}$ via its moments, which will be the sequence given by the sequence $\{C(k,\ve,m)\}$ which we define below. 
\\
\\
\noindent Before we define the sequence, we present some notation we will be using:
\begin{definition}Let $\mathcal T_k$ denote the set of ordered trees with at most $k$ edges. 
\end{definition}
\begin{definition}Let $D =(d_1,\ldots d_t)$ be a sequence of positive integers with $\sum_i d_i = 2(t-1)$. Let $\ve > 0$ be a fixed constant between zero and one. Then define: 
	\begin{equation}\label{psi}
	\psi(D, \ve) =\frac{1}{(2m)^{t-1}}  \lim_{n \rightarrow \infty}\frac{1}{n}\sum_{\ve n \leq x_1 < \ldots < x_t \leq n} \prod_{i = 1}^t \frac{1}{x_i^{d_i/2}} .
	\end{equation}
\end{definition}
\noindent Note the above limit agrees with the earlier definition of $\psi(D,\ve)$:
 \begin{align*}
 \psi(D,\ve) &= \frac{1}{(2m)^{t-1}}  \lim_{n \rightarrow \infty}\frac{1}{n}\sum_{\ve n \leq x_1 < \ldots < x_t \leq n} \prod_{i = 1}^t \frac{1}{x_i^{d_i/2}} \\
 &=  \frac{1}{(2m)^{t-1}}  \lim_{n \rightarrow \infty}\frac{1}{n}(1\pm o(1))\int_{\ve n}^n \int_{\ve n}^{x_t} \ldots \int_{\ve n}^{x_2} \prod_{i=1}^t \frac{1}{x_i^{d_i/2}} dx_1\ldots dx_t\qquad \text{substitute $x_i = ny_i$:}\\
 &=  \frac{1}{(2m)^{t-1}}\left( \int_{\ve}^1 \int_{\ve}^{y_t} \ldots \int_{\ve}^{y_2} \prod_{i=1}^t \frac{1}{y_i^{d_i/2}} dy_1\ldots dy_t\right)\left ( \lim_{n\rightarrow \infty} \frac{(1\pm o(1)) n^{t}}{n \cdot n^{\sum_i d_i/2}}\right)\\
 &=  \frac{1}{(2m)^{t-1}} \int_{\ve}^1 \int_{\ve}^{y_t} \ldots \int_{\ve}^{y_2} \prod_{i=1}^t \frac{1}{y_i^{d_i/2}} dy_1\ldots dy_t.
 \end{align*}
\noindent It is clear that the last formula is a constant depending on $m$, $\ve$,  and $D(T)$. Recall  the critical  constants $C(k,\ve,m)$,  discussed in Section \ref{the approximate limit}: 
\begin{align*}
C(2k,\ve,m) &= \frac{1}{1-\ve} \sum_{T \in \mathcal T_k}  \varphi(T,m)\mathcal M_{2k}(T)\psi(D(T), \ve) \\
C(2k+1, \ve, m) &= 0.
\end{align*}
\noindent Apriori, the above might be a sequence of constants that do not correspond to the moments of any measure. Moreover, if they do correspond to a measure, they might not uniquely define it. However, this is not the case as shown by the following lemma:
\begin{lemma}\label{unique}
	Let $\{C(k,\ve,m)\}_{k=0}^\infty$ be defined as above. Then: 
	\begin{enumerate}
		\item The sequence $\{C(k,\ve,m)\}_{k\geq 0}$ is solvable. 
		\item The sequence is also  $M$-determinate. 
	\end{enumerate}
\end{lemma}

\begin{proof}[Proof of Lemma \ref{unique}]We postpone the proof of part 1. For part 2, note that Carleman's condition requires for the moments to not ``grow too fast". Write  $C_{2k}= C(2k,\ve,m)$ for shorthand. We will show:
	$$
	\limsup_{k\rightarrow \infty}\frac{C_{2k}^{1/2k}}{2k} \leq C
	$$for some constant $C$. 
	We will assume that $\mathcal T_k$ is the set of labeled trees with exactly $k$ edges (as opposed to the set of trees with \textit{at most} $k$ edges, to see why this is enough see Appendix). To bound $\psi(D(T), \ve)$  let $D(T) = (d_1,\ldots, d_{k+1})$. Then:
	\begin{align*} 
	\psi(D(T), \ve) 
	&=\frac{1}{2^k} \lim_{n\rightarrow \infty} \frac{1}{n} \sum_{\ve n\leq x_1< \ldots \leq x_{k+1} \leq n} \prod_{i=1}^{k+1} \frac{1}{x_i^{d_i/2}}\\
	&\leq \frac{1}{2^k} \lim_{n\rightarrow \infty} \frac{1}{n} \prod_{i=1}^{k+1} \sum_{x= \ve n}^n \frac{1}{x^{d_i/2}} 
	\end{align*}Note that
	\begin{equation}\label{sums}
	\sum_{ x= \ve n}^n \frac{1}{x^{d_i/2}} \leq 
	\begin{cases}
	2n^{1-d_i/2} &\text{ if $d_i = 1$} \\
	\log(1/\ve)n^{1-d_i/2} &\text{ if $d_i=2$}\\
	2 n^{1-d_i/2} \ve^{-d_i/2}&\text{ if $d_i\geq 3$}
	\end{cases}
	\end{equation}Thus, 
	\begin{align*}
	\psi(D(T),\ve)
	&\leq \frac{1}{2^k} \lim_{n\rightarrow \infty} \frac{1}{n} \left(\frac{2\log(1/\ve)}{\sqrt{\ve}} \right)^{k+1} n^{k+1 - \sum_i d_i/2} \\
	&=\frac{1}{2^k}\left(\frac{2\log(1/\ve)}{\sqrt{\ve}} \right)^{k+1} \lim_{n\rightarrow \infty} \frac{n^{k+1 -k} }{n} \\ 
	&= \frac{1}{2^k}\left(\frac{2\log(1/\ve)}{\sqrt{\ve}} \right)^{k+1}.
	\end{align*}Hence, $\psi(D(T),\ve) \leq \tilde C^k$ for an appropiate constant $\tilde C$. Hence, we focus on bounding: 
	$$
	\sum_{T\in T_k} \varphi(T,m) \mathcal M_{2k}(T) 
	$$In the Appendix, we show the following combinatorial identity: 
	\begin{align}\label{walk count}
	\mathcal M_{2k}(T) = (2k) \prod_{v\in V(T)} (d(v,T) -1) !
	\end{align}Also note that 
	$$
	\varphi(T,m) \leq \prod_{v\in V(T)} [m]^{d(v, T)}\leq (m+1)^{2k} \prod_{v\in V(T)} (d(v,T)-1)!
	$$Hence, we can bound the terms by functions that only depend on the degree sequence of $T$. Let $\mathcal D$ denote the set of degree sequences of trees on $k+1$ vertices. Then, 
	$$
	\sum_{T\in T_k} \varphi(T,m) \mathcal M_{2k}(T) \leq 2k(m+1)^{2k}\sum_{D \in \mathcal D}|\{ T\in \mathcal T_k : D(T) = D\}| \left(\prod_{d_i \in D} (d_i - 1)! \right)^2
	$$
	By Lemma \ref{moon}, 
	$$
	|\{ T\in \mathcal T_k : D(T) = D\}| = { k-1 \choose d_1-1, d_2-1, \ldots, d_{k+1}-1}
	$$Thus, 
	\begin{align}
	\sum_{T\in \mathcal T_k} \mathcal M_{2k}(T)\varphi(T,m)
	&\leq 2k(m+1)^{2k}(k-1)! \sum_{D\in \mathcal D} \prod_{d_i\in D} (d_i-1)! \nonumber\\
	&= 2k(m+1)^{2k}((k-1)!)^2 \sum_{D\in \mathcal D} \frac{1}{{ k-1 \choose d_1-1, d_2-1, \ldots, d_{k+1}-1}} \nonumber\\ 
	&\leq 2k(m+1)^{2k}((k-1)!)^2 |\mathcal D|\nonumber\\
	&= 2k(m+1)^{2k}((k-1)!)^2 {2k \choose k+1}\nonumber\\
	&\leq 2k(2(m+1))^{2k} ((k-1)!)^2\nonumber\\
	&=2k(m+1)^{2k}((k-1)!)^2\nonumber\\
	&\leq 2k(2k(m+1))^{2k}.\label{final bound}
	\end{align}Hence, we can bound the desired limit by $\sqrt{\tilde{C}}(m+1)$ which proves the lemma. 
	
\end{proof}
\begin{theorem}\label{main1}Let $0<\ve <1$ be fixed. Let $G_{\ve,m,n}$ be the $\ve$-truncation of a random PA-graph. Let $\{C(k,\ve,m)\}$ be the sequence of constants as defined above. Then for any fixed $k$ we have: 
	\begin{align*}
	\BE\left[ \bar{\mu}_{\ve,m,n}^k\right] \longrightarrow C(k,\ve,m).
	\end{align*}
\end{theorem}
\begin{proof}
	We need to compute 
	$$ \BE[\Tr(A_{G_{\ve,m,n}}^k)]$$
	From equation \eqref{muk}, we see that this is equivalent to 
	\begin{equation}\label{e4}
	\BE[\Tr(A_{G_{\ve,m,n}}^k)] = \frac{1}{n-\ve n}\sum_{H\in \mathcal F(K_{m,n-\ve n})} \mathcal M_k(H) \BE[X(H, G_{\ve,m,n})]
	\end{equation}
	Let $H\in \mathcal F(K_{m,n-\ve n})$ be a graph on $t$ vertices. Let $V = \{x_1<\ldots <x_t\}$ be a subset of $[\ve n, n]$. Then by equation \eqref{prob 2} we have: 
	$$
	\BE[X(H, V)] = \left( 1 \pm O\left(\frac{1}{\ve n}\right) \right) \frac{\varphi(H,m)}{(2m)^{|E(H)|}} \prod_{i=1}^t \frac{1}{x_i^{d(v_i,H)/2}}
	$$
	Hence, 
	$$
	\BE[X(H, G_{\ve,m,n})] = \Theta\left(\sum_{\ve n\leq x_1<\ldots<x_t\leq n} \prod_{i=1}^t \frac{1}{x_i^{d(v_i,H)/2}}\right) = O \left(\prod_{i=1}^t \sum_{x = \ve n}^n \frac{1}{x^{d(v_i,H)/2}}\right)
	$$
	Using the same bounds as in equation \eqref{sums}, we obtain: 
	$$
	\BE[X(H, G_{\ve,m,n})] = O\left(n^{t - \sum_i d(v_i,H)/2}\right)= O\left(n^{t -|E(H)|}\right)
	$$
	Since $H$ is a connected graph on $t$ vertices we have $t- |E(H)| \leq 1$, with equality if and only if $H$ is isomorphic to a tree. Hence, in equation \eqref{e4} if we consider the sum over graphs which are not isomorphic to trees we will get a contribution of $O(1/n)$. In particular, note that all the odd moments vanish (since they must contain a cycle). Note that if $H$ is an ordered tree on $t+1$ vertices then we have: 
	\begin{align*}
	\frac{1}{n}\BE[X(H,G_{\ve,m,n})] &= \frac{1}{n}(1 \pm o(1))\sum_{\ve n \leq x_1<\ldots < x_{t+1}\leq n} \frac{\varphi(H,m)}{(2m)^{t}}\prod_{i=1}^{t+1} \frac{1}{x_i^{d(v_i,H)/2}}\\
	&\rightarrow \varphi(H,m) \psi(D(H),\ve)
	\end{align*}
	So to compute the even moments, we obtain:
	\begin{align*}
	\BE[\Tr(A^{2k}_{G_{\ve,m,n}})] 
	&= o(1)+ \frac{1}{1-\ve} \sum_{H \in \mathcal T_k}\mathcal M_{2k}(H)\frac{1}{n} \BE[X(H,G_{\ve,m,n})]\\
	&\rightarrow \frac{1}{1-\ve}\sum_{H\in \mathcal T_k} \mathcal M_{2k}(H)\varphi(H,m) \psi(D(H),\ve)\\
	&= C(2k, \ve, m)
	\end{align*}as desired.
\end{proof}
\noindent Now that we have shown that our sequence of constants correspond to the limiting moments of the expected ESD, then we have that Hamburguer's criterion must be satisfied: The sequence $\{\bar{\mu}_{\ve,m,n}\}_{n=1}^\infty$ satisfies equation \eqref{hamb}. In particular, it follows that their limit, $C(k,\ve, m)$, also satisfies \eqref{hamb}. Since we already showed that the sequence is $M$-determinate, we have: 

\begin{coro}There exists a unique probability measure with moments equal to $\{C(k,\ve,m)\}_{k\geq 0}$. 
\end{coro}

We denote this measure by  $\mu_{\ve,m,\infty}$.
\noindent Hence, to finish the proof of Theorem \ref{theorem:limit}, we need to show that $\{\mu_{\ve,m,n}\}$ converge weakly in probability to $\mu_{\ve,m,\infty}$. We do this by showing  that the random measures are concentrated around the expected measure. 
\begin{lemma}\label{main2}Let $G_{\ve, m,n}$ be an $\ve$-truncated PA random graph, and denote its ESD by $\mu_{\ve,m,n}$. Then for any fixed $k, \delta>0$ 
	$$	
	\lim_{n\rightarrow \infty} \BP\left[ \left|\BE[\mu_{\ve,m,n}^k] - \BE[\bar\mu_{\ve,m,n}^k]\right|> \delta \right] = 0
	$$
\end{lemma}
\begin{proof}We compute the second moment as follows:
	\begin{align*}
	\BE\left[ (\BE[\mu_{\ve,m,n}^k] - \BE[\bar\mu_{\ve,m,n}^k])^2 \right] &= \BE\left[\left( \BE [\mu_{\ve,m,n}^k]\right)^2\right] - \left(\BE[\bar\mu_{\ve,m,n}^k]\right)^2\\
	&= \BE\left[\frac{1}{(n-\ve n)^2} \left(\Tr(A^k_{G_{\ve,m,n}})\right)^2\right] - \left( \frac{1}{n-\ve n} \BE\left[\Tr(A^k_{G_{\ve,m,n}})\right] \right)^2\\
	&= \frac{  1  } {(n-\ve n)^2} \left(\BE\left[\left(\sum_{H\in \mathcal C(G_{\ve,m,n})}\mathcal M_k(H) \mathbb I_{H}\right)^2\right] -\left(\sum_{H\in \mathcal C(G_{\ve,m,n})} \mathcal M_k(H) \BE[\mathbb I_H]\right)^2 \right)\\
	&=\frac{  1  } {(n-\ve n)^2}\sum_{H_1,H_2\in \mathcal C(G_{\ve,m,n})}\mathcal M_k(H_1) \mathcal M_k(H_2)\left(\BE[\mathbb I_{H_1}\mathbb I_{H_2}] - \BE[\mathbb I_{H_1}]\BE[\mathbb I_{H_2}]\right)
	\end{align*}If $H_1$ and $H_2$ are vertex disjoint, then by Lemma \ref{neg corr} we have $$\BE[\mathbb I_{H_1}\mathbb I_{H_2}] - \BE[\mathbb I_{H_1}]\BE[\mathbb I_{H_2}]
	\leq 0$$If $H_1$ and $H_2$ are not vertex disjoint, then $\mathbb I_{H_1}\mathbb I_{H_2}$ corresponds to the indicator random variable that a connected labeled graph appears. As we saw in the above proof, connected graphs give a contribution of $O(n)$. Thus,
	\begin{align*}
	& \sum_{H_1\cap H_2\neq \emptyset}  \mathcal M_k(H_1) \mathcal M_k(H_2)\left(\BE[\mathbb I_{H_1}\mathbb I_{H_2}] -\BE[\mathbb I_{H_1}]\BE[\mathbb I_{H_2}]\right) \\
	&\leq \sum_{H_1\cap H_2\neq \emptyset}  \mathcal M_k(H_1) \mathcal M_k(H_2)\BE[\mathbb I_{H_1}\mathbb I_{H_2}] \\
	&= O(n).
	\end{align*} We thus obtain
	$$
	\BE\left[ (\BE[\mu_{\ve,m,n}^k] - \BE[\bar\mu_{\ve,m,n}^k])^2 \right] = O(n^{-1})
	$$which implies the lemma  via Chebyshev's inequality. 
\end{proof}
\noindent We finish off the proof of Theorem \ref{theorem:limit}:

\begin{proof}By Lemma \ref{conv in moments}, it is enough to show that the $k$-th moments of $\{\mu_{\ve,m,n}\}$ converge in probability to the $k$-th moments of $\mu_{\ve,m,\infty}$. By Theorem \ref{main1}: $$|\BE[\bar\mu_{\ve,m,n}^k - C(k,\ve,m)|\rightarrow 0$$ deterministically. By Lemma \ref{main2}:
	$$
	|\BE[\mu_{\ve,m,n}^k] - \BE[\bar\mu_{\ve,m,n}^k]| \rightarrow 0 
	$$in probability. We get, 
	$$
	|\BE[\mu_{\ve,m,n}^k] - C(k,\ve,m)|\rightarrow 0
	$$in probability as desired.

\end{proof}

\section{Edge Eigenvalues}\label{large evals}
\noindent In this section we prove Theorem \ref{edge1}. As we saw in previous sections, the main contribution towards the even moments come from stars where the high degree appears first. It turns out, that these stars will also dictate the behavior of the largest eigenvalues. 
\\
\\
\noindent In order to show Theorem \ref{edge1}, we will need control on the degrees, in particular we need to show they do not deviate too much from their mean. One could prove these results using the tools developed above, but fortunately some of the bounds needed were already done in \cite{FFF}, so rather than reprove them we will simply state them. 
\\
\\
For a graph $G$, we will write $\|G\|$ for the operator norm of the adjacency matrix. Note that this quantity is always equal to $\lambda_1(G)$ by the Perron-Frobenius Theorem. 
\begin{theorem}[Theorem 1.1 in \cite{FFF}]\label{delta1}Let $t\rightarrow \infty$. Then whp we have $\Delta_1(G_{m,t}) \leq \sqrt{t} \log t$.
\end{theorem} 
\begin{theorem}[Theorem 1.2 in \cite{FFF}]\label{pa norm}Let $t\rightarrow \infty$. Then whp we have $\|G_{m,t}\| \leq t^{1/4} \sqrt{\log t}$. 
\end{theorem}
\begin{theorem}[Claim 2.6 in \cite{FFF}]\label{ub deg}Let $\delta >0$ be fixed. Then for any $s:=s(n)$ such that $s\rightarrow \infty $ as $n\rightarrow \infty$ we have that for all $v\in [s,n]$ and $r\in[s,v]$ we have $d(r, G_{m,v}) \leq n^{\delta}\sqrt{v/r} $ with probability at least $1- O(s^{-C})$ for any arbitrary constant $C$. 
\end{theorem}
\begin{remark}
The above theorem is stated slightly differently in \cite{FFF}. To obtain the form stated above, change $\ell$ (as in their proof) to be an integer satisfying $2-\delta \ell \leq -C$, and then keep track of the error probability. 
\end{remark}
\noindent Lastly, we will be using a result from \cite{PRR} which gives us explicit rates of convergence from the degrees to their limiting distribution. To describe this limiting distribution, we first define $$\kappa_s(x) = \Gamma(s)\sqrt{\frac{2}{s\pi}}\exp\left(\frac{-x^2}{2s}\right)U\left(s-1, \frac{1}{2},\frac{x^2}{2s}\right)$$where $\Gamma(s)$ denotes the gamma function, and $U(a,b,z)$ denotes the confluent hypergeometric function of the second kind. For more information about the above distribution and its properties, see \cite{PRR} and references therein. Denote by $d_K(\mu_1, \mu_2)$ to be the Kolmogorov distance between two probability measures defined as: 
\begin{align*}
d_K(\mu_1,\mu_2) = \sup_{x\in \mathbb R} \left| \mu_1(-\infty, x], - \mu_2(-\infty, x]\right|
\end{align*}
\begin{theorem}[Theorem 1.1 in \cite{PRR}] Let $b_{N,i} = \sqrt{\BE[d(i,G_{1,N})^2]}$. For $1\leq i \leq N$ and some constants $c, C>0$ independent of $N$, we have: $$\frac{c}{\sqrt{N}} \leq d_K \left( \frac{d(i,G_{1,N})}{b_{N,i}}, \kappa_{i-1/2}\right)\leq \frac{C}{\sqrt{N}}$$
\end{theorem}
\noindent In particular, the only property we will need from $\kappa_{i-1/2}(x)$ is that they are uniformly bounded: For all positive integers $i$, we have $\kappa_{i-1/2}(x) \leq 2$ for all $x\in [0,\infty)$. We present our final tool, 
\begin{theorem}\label{deg bounds}Let $\omega(n)\rightarrow \infty$. Then, for all $i\in [1,k]$ we have
$$
d(i, G_{m,n}) \in \left[ \frac{\BE d(i,G_{m,n})}{\omega(n)}, \omega(n) \BE d(i,G_{m,n})\right]
$$with probability at least $1 - O ( k/\omega(n) + k/ \sqrt{n})$. 
\end{theorem}
\begin{proof}A calculation in the Appendix shows 
\begin{align}
\BE[d(i,G_{1,n})] &= \Theta\left(\sqrt{\frac{n}{i}}\right)\label{exp deg}\\
b_{n,i} &= \Theta\left(\sqrt{\frac{n}{i}}\right).\label{2nd moment}
\end{align}Note that $d(i,G_{m,n}) = \sum_{j = (i-1)m +1}^{im}d(j, G_{1,mn})$, so we obtain $\BE d(i, G_{m,n}) = \Theta(\sqrt{n/i})$. Let $\tilde C$ be an absolute constant such that for any $i\in [n]$ and any $j\in [(i-1)m +1, im]$ we have: 
$$
\frac{\BE d(i, G_{m,n})}{b_{mn,j}} \leq \tilde C.
$$
For any fixed function $\omega(n)$, we have:
\begin{align*}
\BP\left[d(i,G_{m,n})\leq \frac{\BE d(i,G_{m,n})}{\omega(n)}\right] 
&\leq  \sum_{j = (i-1)m +1}^{im}\BP\left[d(j, G_{1,mn})\leq \frac{\BE d(i,G_{m,n})}{\omega(n)}\right]\\
&= \sum_{j = (i-1)m +1}^{im}\BP\left[\frac{d(j, G_{1,mn})}{b_{mn,j}}\leq \frac{\BE d(i,G_{m,n})}{b_{mn,j}\omega(n)}\right]\\
&\leq \sum_{j = (i-1)m +1}^{im}\BP\left[\frac{d(j, G_{1,mn})}{b_{mn,j}}\leq \frac{\tilde C}{\omega(n)}\right]\\
&\leq \sum_{j = (i-1)m +1}^{im} \frac{C}{\sqrt{mn}} + \int_0^{\frac{\tilde C}{\omega(n)} }\kappa_{j- 1/2}(x) dx\\
&= O \left( \frac{1}{\sqrt{n}} + \frac{1}{\omega(n)}\right)
\end{align*}Hence, with probability at least $1- O  \left( \frac{k}{\sqrt{n}} + \frac{k}{\omega(n)}\right)$ we have that for all $i \leq k$: $$d(i,G_{m,n})\geq\frac{\BE d(i,G_{m,n})}{\omega(n)}.$$For the upper bound one can just apply Markov's to obtain:
$$
\BP\left[ d(i,G_{m,n}) \geq \omega(n) \BE d(i,G_{m,n})\right]\leq \frac{1}{\omega(n)}.
$$Taking a union bound gives the desired result. 
\end{proof}
\subsection{Proof of Theorem \ref{edge1}} 
\noindent The proof will be similar in nature to the one done in \cite{CLV}. We are going to split the graph into multiple subgraphs, and we are going to show that the eigenvalues can be deduced from a very well structured subgraph, and the rest can be viewed as ``noise". For ease of notation, we will simply write $G$ for $G_{m,n}$. Throughout the proof, we will have the following notation 
$$
\delta = 1/10000 \qquad k = n^{1/25} \qquad b = n^{1/20} \qquad s = n^{1/7} \qquad t= n^{13/25}.
$$
\noindent Let $S = [1,s]$ and $T = [t,n]$. Consider the following subgraphs of $G$ defined as follows: 
\begin{itemize}
\item $G_1= G[1,t]$. 
\item $G_2= G[s, n]$. 
\item $G(S,T)$ is the bipartite graph spanned by the vertex partition $[1,s]$ and $[t,n]$. 
\item $G_3$ is the subgraph of $G(S,T)$ spanned by the vertices in $T$ which have more than one neighbor in $S$. If we have a vertex $v\in T$ which connects to a $u\in S$ through two distinct edges, we consider $t$ having two neighbors in $S$, so $G_3$ will contain all the parallel edges of $G(S,T)$.  
\item $G_4$ is the complement of $G_3$ in $G(S,T)$.
\end{itemize}
\noindent The advantage of the decomposition above is that we obtain $G_4$ which is very well structured: It is a union of $s$ disjoint stars, where we have one star for every $u\in S$. Hence, for $i\leq k$ we will have: $$\lambda_i(G_4) = \sqrt{\Delta_i(G_4)}.$$
We will show that the eigenvalues for $G$ can be deduced from the eigenvalues of $G_4$. This consists of two parts: 
\begin{enumerate}
\item The subgraphs $G_1,G_2$ and $G_3$ have negligible norm.
\item For a vertex $u\in S$, the degree in $G_4$ is close to the degree in $G$. That is, $d(u, G_4)\approx d(u,G)$. 
\end{enumerate}Thus, we define the ``loss" of a vertex $u\in S$ as: 
$$
L(u) = d(u, G_1) + d(u,G_3).
$$Note that we obtain the following relation:  
$$
d(u, G) = d(u, G_4) + L(u).
$$
\noindent Once we have shown that the norms of $G_1,G_2,G_3$ are negligible, we can apply Weyl's inequality: 
how the required bounds through the following five steps:
\begin{align*}
\text{Step 1:} & \text{ Whp we have that for all $i\leq k$:  $d(i, G) \geq \BE[d(i,G)]/ b$.}\\
\text{Step 2:} & \text{ Whp we have $\|G_1\| \leq t^{1/4} \log t$.}\\
\text{Step 3:} & \text{ Whp we have $\|G_2\| = O( n^\delta (n/s)^{1/4})$. }\\
\text{Step 4:} & \text{ Whp we have that for all $u\in S$: $L(u) \leq s^{1.5} \log^2 n + t^{1/2} \log t$.}\\
\text{Step 5:} & \text{ Whp we have $\|G_3\| \leq s^{1.5} \log^2 n$. }
\end{align*}
\noindent Assuming Steps 1 - 5, the proof of Theorem \ref{edge1} follows readily:
\begin{proof}[Proof of Theorem \ref{edge1}]
Recall that $$\BE[d(i,G)] = \Theta\left(\sqrt{\frac{n}{i}} \right) $$which implies by Step 1 that whp we have: 
$$ 
\Delta_i(G) \geq \frac{C}{b}\sqrt{\frac{n}{i}}\geq \frac{C}{b}\sqrt{\frac{n}{k}}
$$for some appropriate constant $C$. 
Note that by Step 4 and the above lower bound, we have that whp for all $i\leq k$: 
\begin{align}
\Delta_i(G_4) 
&\geq \Delta_i(G) - \max_{u\in S} L(u)\nonumber\\
&\geq \Delta_i(G) - (s^{1.5} \log^2 n + t^{1/2} \log t)\label{lb delta}\\
&= (1-o(1))\Delta_i(G).\nonumber
\end{align}
As $\Delta_i(G) \geq \Delta_i(G_4)$, we have $\Delta_i(G_4) = (1-o(1))\Delta_i(G)$. Since $G_4$ is a union of disjoint stars, we have that $\lambda_i(G_4) = \sqrt{\Delta_i(G_4)}$, which means:
$$ \lambda_i(G_4) = (1- o(1)) \sqrt{ \Delta_i(G)}.$$By using our lower bound on $\Delta_i(G)$ and the fact that $i\leq k$, we have: 
$$\lambda_i(G_4)= (1-o(1))\sqrt{\Delta_i(G)}\geq (1-o(1)) \sqrt{\Delta_k(G)} \geq \frac{C}{2}\left(\frac{n^{1/4}}{b^{1/2}k^{1/4}}\right)=n^{43/200}.$$
By Weyl's inequality we have: $$ |\lambda_i(G_4) - \lambda_i(G) | \leq \|G_1\|+\|G_2\|+\|G_3\|$$which using Steps 2, 3 and 5, we have that whp: 
\begin{align}
\|G_1\|+\|G_2\|+\|G_3\| 
&\leq  t^{1/4} \log t +O(n^\delta (n/s)^{1/4})+s^{1.5}\log^2 n\nonumber\\
&= O(n^{\delta} n^{6/28})\nonumber \\
&= o\left( \left(\frac{1}{b}\sqrt{\frac{n}{k}}\right)^{1/2}\right)\label{Cn}\\
&= o(\lambda_k(G_4)).\nonumber
\end{align}Thus,
$$
|\lambda_i(G_4) - \lambda_i(G) | = o(\lambda_i(G_4))
$$
Combining all of the above, we obtain that whp for all $i\leq k$: 
\begin{align*} 
\lambda_i(G) 
&= (1\pm o(1)) \lambda_i(G_4) \\ 
&= (1\pm o(1)) \sqrt{\Delta_i(G_4)}\\
&= (1\pm o(1)) \sqrt{\Delta_i(G)} 
\end{align*}just as desired. 
\end{proof}
\noindent We now prove each one of the required Steps. 
\\
\\
\textbf{Step 1: } Follows from Theorem \ref{deg bounds} by letting $\omega(n) = b$. 
\\
\\
\textbf{Step 2: } Follows from Theorem \ref{pa norm}.
\\
\\
\textbf{Step 3: }We will be using equation \eqref{trick}. Our matrix $A$ will be the adjacency matrix for $G_2$, where for ease of notation the rows and columns are labeled $[s,n]$. Let $r = \sqrt{ns}$, $c_i = \frac{n^{1/4}}{i^{1/4}}$ for $i \in [s,r]$, and $c_i = 1$ for $i \in (r,n]$. Define, 
\[
R_i = \frac{1}{c_i}\sum_{j=s}^n c_j a_{ij}.
\]It suffices to show that for all $i \in [s,n]$ we have $R_i =O(n^\delta (n/s)^{1/4})$ with probability $1-o(n^{-1})$, so the result would follow from taking a union bound. Let $C$ be a large enough constant such that $s^{-C} = o(n^{-1})$. Then, using Theorem \ref{ub deg} we have:
\\
\\
\textbf{Case $i \in [s,r]$ : }We have 
\begin{align*}
R_i
& = \frac{1}{c_i}\left(\sum_{j=s}^i c_j a_{ij} + \sum_{j = i+1}^r c_j a_{ij} +\sum_{j = r+1}^n c_j a_{ij}\right)\\
&\leq \frac{mc_s}{c_i}+d(i, G_{m,r})+ \frac{d(i, G_{m,n})}{c_i}\\
&\leq \frac{mc_s}{c_i}+\sqrt{\frac{r}{i}} r^\delta+\frac{1}{c_i} \sqrt{\frac{n}{i}} n^\delta \\
&= O\left(\left(\frac{n}{s}\right)^{1/4}n^\delta\right)
\end{align*}occurs with probability $1-O(s^{-C})= 1-o(n^{-1})$. 
\\
\\
\textbf{Case $i\in (r,n]$ : }In this case one has: 
\begin{align*}
R_i 
&= \sum_{j=s}^i c_ja_{ij} + \sum_{j= i+1}^n a_{ij}\\
&\leq m c_s+ d(i, G_{m,n}) \\ 
& = m \sqrt[4]{\frac{n}{s}} + \sqrt{\frac{n}{r}} n^\delta\\ 
& = O \left( \left( \frac{n}{s}\right)^{1/4} n^\delta \right)
\end{align*}occurs with probability $1-O(r^{-C})= 1-o(n^{-1})$. Hence, whp we have $R_i= O(n^\delta (n/s)^{1/4})$ for all $i\in [s,n]$. Thus, whp we have: $$\|G_2\| = O \left( \left( \frac{n}{s}\right)^{1/4} n^\delta \right).$$
\\
\\
\textbf{Step 4: }For $u\in S$, we have: 
\begin{equation}\label{ub1}
d(u, G_1)\leq \Delta_1(G_1) \leq \sqrt{t} \log t
\end{equation}
whp by Theorem \ref{delta1}. Let $H_1$ be the following ordered graph $\{(v_1,v_3), (v_2,v_3)\}$, and let $H_2 = \{(v_1,v_2),(v_1,v_2)\}$. Then, 
$$
d(u, G_3) \leq \sum_{x_2 \in S, x_3\in T} X_{m,n}(H_1, \{u, x_2, x_3\}) + 2 \sum_{x_2\in T} X_{m,n}(H_2,\{u, x_2\}).
$$By computing the above graph counts (see Appendix) we obtain: 
\begin{align}\label{d u g3}
\BE[d(u,G_3)] = O(\sqrt{s} \log n)
\end{align}Hence, for an appropriate constant $C$ (independent of $n$), we have: 
\begin{align*}
\BP[d(u, G_3) \geq s^{1.5} \log^2 n ] 
&\leq \BP[d(u, G_3)\geq C s \log n \BE[d(u,G_3)]] \\
&\leq \frac{1}{Cs \log n}.
\end{align*}By taking a union bound over all $u\in S$, we obtain that with high probability: 
$$
d(u, G_3) \leq s^{1.5} \log^2 n.
$$Combining with equation \eqref{ub1} we obtain: 
$$
L(u) \leq  s^{1.5} \log^2 n + \sqrt{t} \log t.
$$
\textbf{Step 5: }We will use the fact that the norm of a graph is always bounded by the maximum degree. Recall, $G_3$ is a bipartite graph where for every $v\in T$ we have that $d(v, G_3) \leq m$ (by definition of our model), and for $u\in S$ we have $d(u, G_3) \leq s^{1.5} \log^2 n$ (by above). Hence, whp: 
$$
\|G_3\| \leq s^{1.5} \log^2 n .
$$This concludes the proof of Theorem \ref{edge1}. Now we proceed to study the largest eigenvectors. 

\section{Localization of Eigenvectors} 
\noindent In this section, we prove Theorem \ref{evec}. The strategy is as follows: We will use the decomposition and all the notation we used in the previous section (e.g., $k, t, G_1, G_2,\ldots$). Let $G_{m,n}$ be a random PA graph, and let $A_n$ denote its adjacency matrix. Let $B_n$ denote the adjacency matrix of $G_4$, and let $C_n := A_n - B_n$. From the Davis-Kahan Theorem (Theorem \ref{dk}), we see that the two quantities one needs to control when studying eigenvectors of perturbations are the eigenvalue gap and the size of the noise. We present the following bounds on the two quantities:
\begin{lemma}\label{Cn lemma}Let $C_n$ be defined as above. Then whp:
$$\|C_n\| = O(n^{43/200}).$$
\end{lemma}
\begin{proof}
$C_n$ corresponds to the adjacency matrix of the complement of $G_4$ in $G_{m,n}$. As we saw in the proof of Theorem \ref{edge1}, we see that this complement is contained in $G_1\cup G_2 \cup G_3$. Hence, by equation \eqref{Cn} we obtain whp  
\begin{align*}
\|C_n\| &\leq \|G_1\| + \|C_3\| + \|G_3\| \\ 
&= O(n^{43/200}).
\end{align*}
\end{proof}
\noindent Recall that the eigenvalues of $B_n$ arise from the (square root of) max degrees. Hence, to control the eigenvalue gap, we need to control the maximum degree gap. We use the following theorem: 
\begin{theorem}[Theorem 1 in \cite{FFF}]Let $K$ be some fixed constant. We have whp: 
$$
\frac{\sqrt{n}}{\log n} \leq \Delta_i(G) \leq \Delta_{i-1}(G) - \frac{\sqrt{n}}{\log n}
$$
$$
\frac{\sqrt{n}}{\log n} \leq \Delta_i(G).
$$
\end{theorem}
\noindent We can now show a lower bound on the eigenvalue gap: 
\begin{lemma}\label{gap}Using the notation from above, for $i\leq K$ we have whp: 
$$
\lambda_i(B_n) - \lambda_{i+1}(B_n) \geq \frac{n^{1/4}}{\log^3 n}.
$$
\end{lemma}

\noindent For a proof of the above lemma, see Appendix. With these two bounds, we can present the proof. 

\begin{proof}[Proof of Theorem \ref{evec}]Since $B_n$ is the union of disjoint stars we have that for $i\leq K$ whp: 
$$
\|w_i\|_\infty = \frac{1}{\sqrt{2}}
$$and all coordinates which do not realize the infinity norm are equal to $1/\sqrt{2\Delta_i(G_4)} = o(1)$. Hence, if we show: 
$$
\|v_i - w_i\|_2 = o(1)
$$then the proof  would follow. First of all note that for two unit vectors $u$ and $v$, whose angle $\theta$ is in $[0,\pi/2]$ we have: 
\begin{align*}
\|u-v\|_2^2 
&\leq 2\sin^2(\theta).
\end{align*}Hence, it suffices to show that the (sine of the) angle between $w_i$ and $v_i$ is $o(1)$. By Lemmas \ref{Cn lemma} and \ref{gap}: 
\begin{align*}
\|C_n\| &= O(n^{43/200}) \\
\lambda_i(B_n) - \lambda_{i+1}(B_n) &\geq \frac{n^{1/4}}{\log^3 n}
\end{align*}occur whp for all $i\leq K$. Hence, 
\begin{align*}
|\lambda_{i-1} (A_n)- \lambda_i(B_n)| &=|\lambda_{i-1} (A_n) - \lambda_{i-1}(B_n) + \lambda_{i-1}(B_n) -\lambda_i(B_n)| \\
&\geq \frac{n^{1/4}}{2\log^3 n}.
\end{align*}
Similarly, we can bound $|\lambda_{i}(B_n) - \lambda_{i+1}(A_n)|$ from below by the same quantity. Thus, by Theorem \ref{dk}:
$$
\sin\angle(v_i,w_i) = O\left(\frac{\log^3 n }{n^{7/200}}\right)
$$which implies the desired result. 

\end{proof}

\section{Appendix}\label{appendix}
\noindent \textbf{Proof of lemma \ref{moments lemma}: }Note that we have that the sequence $\{\mu_n\}$ is tight. Hence, by Prokhorov's theorem there exists a subsequence $\{\mu_{n_j}\}_{j=1}^\infty$ such that it converges weakly to a measure, call it $\tilde \mu$. Note that the moments of $\tilde \mu$ are equal to the moments of $\mu$, but since they uniquely determine the measure, we have $\tilde \mu = \mu$. By Prokhorov's theorem, since every weakly convergent subsequence has the same limit $\mu$, then the sequence $\{\mu_n\}$ converges weakly to $\mu$. 
\\
\\
\textbf{Proof of lemma \ref{conv in moments}: }Fix $f$ to be a bounded continuous function. Let $X_n = \int f d\mu_n$ and $X = \int f d\mu$. We want to show that $X_n$ converges in probability to $X$. This is equivalent to showing that for any subsequence $\{n_m\}$ one can find a further subsequence $\{n_{m_j}\}$ such that $X_{n_{m_j}} $ converges almost surely to $X$. 
\\
\\
Given $\{n_m\}$, we extract a subsequence $\{n_{m_j}\}$ such that $\int x^k d\mu_{n_{m_j}}$ converges almost surely to $\int x^k d\mu$ for all $k$: This can be done by first choosing a subsequence in which the first moments converges a.s., then pick a subsequence where the second moment converges a.s., and continue inductively. Since $\mu$ is uniquely determined by its moments, we have that $\mu_{n_{m_j}}$ converges weakly to $\mu$ with probability equal to one. This implies that $X_{n_{m_j}}$ converges almost surely to $X$, just as desired. 
\\
\\
\noindent \textbf{Moreover part of Theorem \ref{mag}:} We wish to lower bound: 
$$  \sum_{1\leq x_1 < \ldots < x_t\leq n} \Theta \left( \prod_{i=1}^t \frac{1}{x_i^{d(v_i,H)/2}}\right) $$
where we have $d(v_i, H)\geq d(v_{i+1},H)$. Let $I_1$ and $I_2$ be such that: $d(v_i, H) \geq 3$ for $i\in [1,I_1]$, $d(v_i, H) = 2$ for $i\in [I_1 + 1, I_2]$ and $d(v_i,H) = 1$ for $i\in [I_2+1, t]$. Then, 
\begin{align*}
\sum_{1\leq x_1 < \ldots < x_t\leq n} \Theta \left( \prod_{i=1}^t \frac{1}{x_i^{d(v_i,H)/2}}\right)  &\geq 
\sum_{I_1 + 1 \leq x_{I_1+1}< \ldots < x_t\leq n} \Theta\left( \prod_{i=I_1+1}^t \frac{1}{x_i^{d(v_i,H)/2}}\right) \\
&\geq \Theta\left(\left( \sum_{I_1 + 1 \leq x_{I_1 + 1}< \ldots < x_{I_2} \leq \sqrt{n}}  \prod_{i = I_1+1}^{I_2} \frac{1}{x_i}\right)\left( \sum_{\sqrt{n} \leq x_{I_2+1}< \ldots<x_t\leq n}   \prod_{i = I_2+1}^{t} \frac{1}{\sqrt{x_i}}\right)\right)\\
&= \Theta(n^{f(H)/2}\log^{g(H)} n).
\end{align*}
\noindent \textbf{Proof of equations \eqref{exp deg} and \eqref{2nd moment}:} We have 
\begin{align*}
\BE[ d(i, G_{1,n})] &=  1 + \sum_{j = i+1}^n \BP[(i,j) \in G_{1,n}]\\
&= 1 + \Theta\left( \sum_{j = i+1}^n \frac{1}{\sqrt{ij}}\right) \\ 
&= \Theta \left(\sqrt{\frac{n}{i}}\right) .
\end{align*}Moreover, 
\begin{align*}
b_{n,i}^2 &= \BE[d(i,G_{1,n})^2] \\ 
&= \BE\left[ \left( 1+ \sum_{j =i+1}^n \mathbb I_{(i,j) \in G_{1,n}}\right)^2 \right] \\
&= 1 + \BE\left[3d(i, G_{1,n})+ 2\sum_{i+1 \leq j < j'\leq n}  \mathbb I_{(i,j) \in G_{1,n}} \mathbb I_{(i,j') \in G_{1,n}}\right] \\
&= 1 + \Theta \left(\sqrt{\frac{n}{i}}\right)  + 2\sum_{i+1 < j < j' \leq n }\BP[(i,j),(i,j') \in G_{1,n}] \\ 
&= 1 + \Theta \left(\sqrt{\frac{n}{i}}\right)  + 2\sum_{i+1 < j < j' \leq n }\Theta\left( \frac{1}{i\sqrt{jj'}}\right)\\ 
&= \Theta\left(\frac{n}{i}\right).
\end{align*}Taking a square root yields equation \eqref{2nd moment}. 
\\
\\
\noindent \textbf{Proof of equation \eqref{d u g3}: }Let $V = \{x_1,x_2,x_3\}$. Then from equation \eqref{prob 2} we obtain the following upper bound: 
$$ 
\BE[X_{m,n}(H_1, V)] = O\left( \frac{1}{\sqrt{x_1x_2} x_3}\right) 
$$Hence,
\begin{align*} 
\sum_{x_2 \in S, x_3\in T} X_{m,n}(H_1, \{u, x_2, x_3\}) &= O\left(\sum_{x_2 \in S, x_3\in T}  \frac{1}{\sqrt{ux_2} x_3}\right) \\
&= O(\sqrt{s} \log n)
\end{align*}
Similarly for $H_2$, if $V= \{x_1,x_2\}$ then: 
$$
\BE[X_{m,n}(H_2, V)] = O \left(\frac{1}{x_1x_2}\right).
$$Hence, 
\begin{align*} 
\sum_{x_2\in T} X_{m,n}(H_2, \{u, x_2\}) &= O\left(\sum_{x_2 \in T}  \frac{1}{u x_2}\right) \\
&= O(\log n)
\end{align*}which imply equation \eqref{d u g3}. 
\\
\\
\noindent \textbf{Proof of Lemma \ref{gap}: }First of all note that $\lambda_i(B_n) = \sqrt{\Delta_i(G_4)}$ for $i\leq k$ whp. Moreover, we also have: 
$$
\left( 1 - O\left( n^{13/50}\log n\right)\right)\Delta_i(G)\leq \Delta_i(G_4)\leq \Delta_i(G)
$$
Thus, 
\begin{align*}
\frac{\Delta_{i+1}(G_4)}{\Delta_i(G_4)} &\leq \frac{\Delta_{i+1}(G)}{\Delta_i(G) - n^{13/50}\log n}\\
&= \frac{1}{1 - \frac{ n^{13/50} \log n}{\Delta_i(G)}} \frac{\Delta_{i+1}(G_4)}{\Delta_i(G_4)}\\
&= \left( 1 + O\left(\frac{1}{n^{12/50}\log n}\right)\right) \left( 1 - O\left(\frac{1}{\log^2 n} \right)\right) \\
&=  \left( 1 - O\left(\frac{1}{\log^2 n} \right)\right) .
\end{align*}
Hence, 
\begin{align*}
\lambda_i(B_n) - \lambda_{i+1}(B_n) &= \sqrt{\Delta_i(G_4)} - \sqrt{\Delta_{i+1}(G_4)} \\
&= \sqrt{\Delta_i(G_4)} \left(1 - \sqrt{\frac{\Delta_{i+1}(G_4)}{\Delta_i(G_4)}}\right)\\
&= \sqrt{\Delta_i(G_4)} \left(1 - \sqrt{ 1 - O\left(\frac{1}{\log^2 n} \right)} \right)\\
&\geq \frac{n^{1/4}}{\log^{1/2} n}\cdot  O\left( \frac{1}{\log ^2 n}\right) \\ 
&\geq \frac{n^{1/4}}{\log^3 n}.
\end{align*}
\\
\\
\noindent \textbf{Proof of equation \eqref{walk count}: }Let $T$ be a tree. We wish to count the number of closed walks in $T$ which use each edge exactly twice (back and forth in the walk). Note that we can rewrite equation \eqref{walk count} as: 
$$
\sum_{v\in T} d(v, T) ! \prod_{w\neq v} (d(w, T) - 1)!
$$We claim that the number of closed walks in $T$ which use each edge exactly twice, starting and ending at $v$ is given by: $$d(v,T)! \prod_{w\neq v} (d(w, T) -1)!$$which follows easily from induction. 
\\
\\
\noindent \textbf{Reduction of $\mathcal T_k$ on the proof of Lemma \ref{unique}: }We want to show that it is enough to consider the case when $\mathcal T_k$ is the set of trees with exactly $k$ edges. If $T$ is a tree on $\tilde k$ many edges, where $\tilde k  < k$, then: 
$$
\mathcal M_{2\tilde k +2}(T) \leq (\Delta(T) (2\tilde k) ) \mathcal M_{2\tilde k}(T) 
$$to see this, note that given a closed walk of length $2\tilde k$, then we can choose a point during the walk to add an extra step. The number of places where the extra step goes is $2\tilde k$, and at each vertex we have at most $\Delta(T)$ many choices on which vertex to visit from there. Hence, 
$$
\mathcal M_{2k}(T) \leq (2k \Delta(T))^{k-\tilde k} \mathcal M_{2\tilde k}(T).
$$Note that calculation \eqref{final bound} can be carried the same with $k$ replaced by $\tilde k$. Hence, when we sum over the trees with exactly $\tilde k$ many edges we get an upper bound of: 
\begin{align*}
(2k \Delta(T))^{k-\tilde k}2\tilde k(2\tilde k(m+1))^{2\tilde k}&\leq (2k^2)^{k-\tilde k} 2 k (2k(m+1))^{2 \tilde k}\\
&\leq C^k k^{2k}
\end{align*}Hence, by summing over all $\tilde k < k$ we get an extra factor of $k$ to obtain: 
$$
C^k k^{2k+1}
$$which satisfies the required limit needed in equation \eqref{condition}.


\end{document}